\documentclass[11pt]{amsart}
\begin{document}

\def\dbl{[\hskip -1pt[}
\def\dbr{]\hskip -1pt]}
\title[Holomorphic versus algebraic equivalence]{Holomorphic versus algebraic equivalence\\  for deformations of real-algebraic CR manifolds}
\author{Bernhard Lamel}
\address{Universit\"at Wien, Fakult\"at f\"ur Mathematik, Nordbergstrasse 15, A-1090 Wien, \"Osterreich}
\email{lamelb@member.ams.org}%
\thanks{The first author is supported by the Austrian Federal Ministry of Science and Research BMWF, START Prize Y377. The second author was partially supported by the French National Agency for Research (ANR), project  DynPDE (programmes blancs). Both authors were partially supported by the Amadeus program of the Partenariat Hubert Curien and the FWF-ANR project CRARTIN}
\author{Nordine Mir}
\address{Universit\'e de Rouen, Laboratoire de Math\'ematiques Rapha\"el Salem, UMR 6085 CNRS, Avenue de
l'Universit\'e, B.P. 12, 76801 Saint Etienne du Rouvray, France}
\email{Nordine.Mir@univ-rouen.fr}

%\date{\number\year-\number\month-\number\day}
%\enddate
%\loadeufm

%\def\Label#1{\label{#1}{\bf \hbox{  }(#1)\hbox{  } }}
\def\Label#1{\label{#1}}
\def\1#1{\ov{#1}}
\def\2#1{\widetilde{#1}}
\def\6#1{\mathcal{#1}}
\def\4#1{\mathbb{#1}}
\def\3#1{\widehat{#1}}
\def\7#1{\widehat{#1}}
\def\K{{\4K}}
\def\LL{{\4L}}

\def \MM{{\4M}}
\def \S{{\4S}^{2N'-1}}

\def \B{{\4B}^{2N'-1}}

\def \H{{\4H}^{2l-1}}

\def \F{{\4H}^{2N'-1}}

\def \LL{{\4L}}

\def\Re{{\sf Re}\,}
\def\Im{{\sf Im}\,}
\def\id{{\sf id}\,}

%\numberwithin{equation}{section}
\def\s{s}
\def\k{\kappa}
\def\ov{\overline}
\def\span{\text{\rm span}}
\def\ad{\text{\rm ad }}
\def\tr{\text{\rm tr}}
\def\xo {{x_0}}
\def\Rk{\text{\rm Rk\,}}
\def\sg{\sigma}
\def \emxy{E_{(M,M')}(X,Y)}
\def \semxy{\scrE_{(M,M')}(X,Y)}
\def \jkxy {J^k(X,Y)}
\def \gkxy {G^k(X,Y)}
\def \exy {E(X,Y)}
\def \sexy{\scrE(X,Y)}
\def \hn {holomorphically nondegenerate}
\def\hyp{hypersurface}
\def\prt#1{{\partial \over\partial #1}}
\def\det{{\text{\rm det}}}
\def\wob{{w\over B(z)}}
\def\co{\chi_1}
\def\po{p_0}
\def\fb {\bar f}
\def\gb {\bar g}
\def\Fb {\ov F}
\def\Gb {\ov G}
\def\Hb {\ov H}
\def\zb {\bar z}
\def\wb {\bar w}
\def \qb {\bar Q}
\def \t {\tau}
\def\z{\chi}
\def\w{\tau}
\def\Z{\zeta}
\def\phi{\varphi}
\def\eps{\epsilon}

\def \T {\theta}
\def \Th {\Theta}
\def \L {\Lambda}
\def\b {\beta}
\def\a {\alpha}
\def\o {\omega}
\def\l {\lambda}

\def \im{\text{\rm Im }}
\def \re{\text{\rm Re }}
\def \Char{\text{\rm Char }}
\def \supp{\text{\rm supp }}
\def \codim{\text{\rm codim }}
\def \Ht{\text{\rm ht }}
\def \Dt{\text{\rm dt }}
\def \hO{\widehat{\mathcal W}}
\def \cl{\text{\rm cl }}
\def \bS{\mathbb S}
\def \bK{\mathbb K}
\def \bD{\mathbb D}
\def \bC{\mathbb C}
\def \bL{\mathbb L}
\def \bZ{\mathbb Z}
\def \bN{\mathbb N}
\def \scrF{\mathcal F}
\def \scrK{\mathcal K}
\def \mc #1 {\mathcal {#1}}
\def \scrM{\mathcal M}
\def \cR{\mathcal R}
\def \scrJ{\mathcal J}
\def \scrA{\mathcal A}

\def \scrV{\mathcal V}
\def \scrL{\mathcal L}
\def \scrE{\mathcal E}
\def \hol{\text{\rm hol}}
\def \aut{\text{\rm aut}}
\def \Aut{\text{\rm Aut}}
\def \J{\text{\rm Jac}}
\def\jet#1#2{J^{#1}_{#2}}
\def\gp#1{G^{#1}}
\def\gpo{\gp {2k_0}_0}
\def\emmp {\scrF(M,p;M',p')}
\def\rk{\text{\rm rk\,}}
\def\Orb{\text{\rm Orb\,}}
\def\Exp{\text{\rm Exp\,}}
\def\Span{\text{\rm span\,}}
\def\d{\partial}
\def\D{\3J}
\def\pr{{\rm pr}}

\def \CZZ {\C \dbl Z,\zeta \dbr}
\def \D{\text{\rm Der}\,}
\def \Rk{\text{\rm Rk}\,}
\def \CR{\text{\rm CR}}
\def \ima{\text{\rm im}\,}
\def \I {\mathcal I}

\def \M {\mathcal M}

\newtheorem{Thm}{Theorem}[section]
\newtheorem{Cor}[Thm]{Corollary}
\newtheorem{Pro}[Thm]{Proposition}
\newtheorem{Lem}[Thm]{Lemma}

\theoremstyle{definition}\newtheorem{Def}[Thm]{Definition}

\theoremstyle{remark}
\newtheorem{Rem}[Thm]{Remark}
\newtheorem{Exa}[Thm]{Example}
\newtheorem{Exs}[Thm]{Examples}
%\numberwithin{equation}

\numberwithin{equation}{section}

\def\bl{\begin{Lem}}
\def\el{\end{Lem}}
\def\bp{\begin{Pro}}
\def\ep{\end{Pro}}
\def\bt{\begin{Thm}}
\def\et{\end{Thm}}
\def\bc{\begin{Cor}}
\def\ec{\end{Cor}}
\def\bd{\begin{Def}}
\def\ed{\end{Def}}
\def\be{\begin{Exa}}
\def\ee{\end{Exa}}
\def\bpf{\begin{proof}}
\def\epf{\end{proof}}
\def\ben{\begin{enumerate}}
\def\een{\end{enumerate}}

\newcommand{\dbar}{\bar\partial}
\newcommand{\genmat}{\lambda}
\newcommand{\polynorm}[1]{{|| #1 ||}}
\newcommand{\vnorm}[1]{\left\|  #1 \right\|}
\newcommand{\asspol}[1]{{\mathbf{#1}}}
\newcommand{\Cn}{\mathbb{C}^n}
\newcommand{\Cd}{\mathbb{C}^d}
\newcommand{\Cm}{\mathbb{C}^m}
\newcommand{\C}{\mathbb{C}}
\newcommand{\CN}{\mathbb{C}^N}
\newcommand{\CNp}{\mathbb{C}^{N^\prime}}
\newcommand{\Rd}{\mathbb{R}^d}
\newcommand{\Rn}{\mathbb{R}^n}
\newcommand{\RN}{\mathbb{R}^N}
\newcommand{\R}{\mathbb{R}}
\newcommand{\bR}{\mathbb{R}}
\newcommand{\N}{\mathbb{N}}
\newcommand{\dop}[1]{\frac{\partial}{\partial #1}}
\newcommand{\vardop}[3]{\frac{\partial^{|#3|} #1}{\partial {#2}^{#3}}}
\newcommand{\br}[1]{\langle#1 \rangle}
\newcommand{\infnorm}[1]{{\left\| #1 \right\|}_{\infty}}
\newcommand{\onenorm}[1]{{\left\| #1 \right\|}_{1}}
\newcommand{\deltanorm}[1]{{\left\| #1 \right\|}_{\Delta}}
\newcommand{\omeganorm}[1]{{\left\| #1 \right\|}_{\Omega}}
\newcommand{\nequiv}{{\equiv \!\!\!\!\!\!  / \,\,}}
\newcommand{\bk}{\mathbf{K}}
\newcommand{\p}{\prime}
\newcommand{\tV}{\mathcal{V}}
\newcommand{\poly}{\mathcal{P}}
\newcommand{\ring}{\mathcal{A}}
\newcommand{\ringk}{\ring_k}
\newcommand{\ringktwo}{\mathcal{B}_\mu}
\newcommand{\germs}{\mathcal{O}}
\newcommand{\On}{\germs_n}
\newcommand{\mcl}{\mathcal{C}}
\newcommand{\formals}{\mathcal{F}}
\newcommand{\Fn}{\formals_n}
\newcommand{\autM}{{\Aut (M,0)}}
\newcommand{\autMp}{{\Aut (M,p)}}
\newcommand{\holmaps}{\mathcal{H}}
\newcommand{\biholmaps}{\mathcal{B}}
\newcommand{\autmaps}{\mathcal{A}(\CN,0)}
\newcommand{\jetsp}[2]{ G_{#1}^{#2} }
\newcommand{\njetsp}[2]{J_{#1}^{#2} }
\newcommand{\jetm}[2]{ j_{#1}^{#2} }
\newcommand{\glnc}{\mathsf{GL_n}(\C)}
\newcommand{\glmc}{\mathsf{GL_m}(\C)}
\newcommand{\glc}{\mathsf{GL_{(m+1)n}}(\C)}
\newcommand{\glk}{\mathsf{GL_{k}}(\C)}
\newcommand{\smC}{\mathcal{C}^{\infty}}
\newcommand{\anC}{\mathcal{C}^{\omega}}
\newcommand{\kC}{\mathcal{C}^{k}}
\newcommand{\fps}[1]{\C[[#1]]}
\newcommand{\cps}[1]{\C\{#1\}}

\newcommand{\dopat}[2]{\frac{\partial}{\partial #1}\biggr|_{#2}}
\newcommand{\dopt}[2]{\frac{\partial #1}{\partial #2}}
% \newcommand{\vardop}[3]{\frac{\partial^{|#3|} #1}{\partial {#2}^{#3}}}
% \newcommand{\br}[1]{\langle#1 \rangle}

%\keywords{}

\maketitle

\begin{abstract} 
We consider (small) algebraic deformations of germs of real-algebraic CR submanifolds in complex space and study the biholomorphic equivalence problem for such deformations. We show that two algebraic deformations of minimal holomorphically nondegenerate real-algebraic CR submanifolds are holomorphically equivalent if and only if they are algebraically equivalent.
\end{abstract}

%\maketitle
%\tableofcontents

\section{Introduction}\Label{int}

Since Poincar\'e's celebrated paper \cite{Po} published in 1907, there has been a growing literature concerned with the equivalence problem for real submanifolds in complex space (see e.g.\ \cite{G, Sto, HY1, HY2, BRZ1,BRZ2, BMR} for some recent works as well as the references therein). One  interesting phenomenon, observed by Webster for biholomorphisms of 
Levi nondegenerate hypersurfaces \cite{We1}, is that the biholomorphic equivalence of some types of 
real-algebraic submanifolds of a complex space implies their algebraic equivalence.
%Similar problems for other types of equivalence
%have attracted considerable attention over the last years, 
%see e.g.\ \cite{G, Sto, HY1, HY2, BRZ1,BRZ2, BMR} and the references therein.

In this paper, we show that this very phenomenon holds for algebraic deformations of germs of  
minimal holomorphically nondegenerate 
real-algebraic CR submanifolds in complex space. 
Let us 
recall that a germ of a real-algebraic CR submanifold $(M,p)\subset (\C^n,p)$ is {\em minimal} if there exists no proper CR submanifold $N\subset M$ 
through $p$ of the same CR dimension as $M$. It is {\em holomorphically nondegenerate} 
if there exists no nontrivial holomorphic vector field tangent to $M$ near $p$ (see \cite{St}). 

An algebraic deformation of $(M,p)$ is a real-algebraic family of germs at $p$ of real-algebraic CR submanifolds  
$(M_s,p)_{s\in\R^k}$ in $\C^n$, defined for $s\in\R^k$ near $0$, such that $M_0 = M$. We say that two 
such deformations $(M_s,p)_{s\in\R^k}$ and $(N_t,p')_{t\in\R^k}$ are {\em biholomorphically equivalent} 
if there exists a germ of a real-analytic diffeomorphism
$\varphi \colon (\R^k,0) \to (\R^k,0)$ and a holomorphic submersion
$B\colon(\C_z^{n}\times \C_u^{k},(p,0)) \to (\Cn,p')$ such that $z\mapsto B(z,s)$ is a biholomorphism sending  
$(M_s,p)$ to $(N_{\varphi(s)},p')$ for all $s\in \R^k$ close to $0$. We shall say that such a pair $(B,\varphi)$ is a biholomorphism between the two deformations\footnote{By slight abuse of language, we shall always identify the map $\phi\colon (\R^k,0) \to (\R^k,0)$ with its complexification from $(\C^{k},0)$ to $(\C^k,0)$.}. We also say that they are {\em algebraically equivalent} if one can choose $\varphi$ and $B$ to be furthermore algebraic.
Our main result is as follows.

\begin{Thm}
	\label{thm:main2} Two algebraic deformations of minimal holomorphically nondegenerate real-algebraic CR submanifolds of $\C^n$ are algebraically equivalent if and only if they are biholomorphically equivalent.
\end{Thm}

For a completely trivial deformation (i.e. $k=0$), 
Theorem~\ref{thm:main2} is actually a consequence of the algebraicity theorem of
Baouendi, Ebenfelt and Rothschild \cite{BER96}, where they prove that
every local biholomorphism sending holomorphically nondegenerate and minimal real-algebraic generic submanifolds of $\C^n$ must necessarily be algebraic. This is not necessarily true for biholomorphisms between deformations, even for constant 
ones, as the following example shows.

\begin{Exa}\label{ex:1}
	Consider the Lewy hypersurface $M$ in $\C^2_{(z,w)}$ defined by 
	$\im w = |z|^2$, and consider the trivial deformation $M_s = M$ for $s\in \R^k$.  A biholomorphic 
	map of $(M_s,0)_{s\in\R^k}$ 
	to itself which is not algebraic 
	is e.g. given by $\varphi(s) = s$,   
	 $B(z,w,s)= (e^s z, e^{2s} w)$.
\end{Exa}

The main point of this example is that one cannot expect a  biholomorphism between two deformations  
to be algebraic. However, in Example~\ref{ex:1}, all the "fibers" $M_s$ of the deformations are
the same. It is not difficult to show, 
by using the mentioned result of \cite{BER96}, that the conclusion of Theorem~\ref{thm:main2} 
holds when all the fibers of the deformation are algebraically equivalent. 
One approach which has been successful for more general deformations 
% (actually, 
% deformations of a special class of 
% minimal CR manifolds, the  so-called ``{finitely nondegenerate}'' ones, see \cite{BRZ2}),  
is to approximate a given 
biholomorphism by algebraic ones;  Theorem~\ref{thm:main2} is a consequence of 
such an approximation statement.

\begin{Thm}
	\label{thm:main2app} Let $(M_s,p)_{s\in\R^k}$ and $(N_t,p')_{t\in\R^k}$ be 
	algebraic deformations of real-algebraic holomorphically nondegenerate minimal
	CR submanifolds of $\C^n$, and assume that $(B,\varphi)$ is a biholomorphism between
	 $(M_s,p)_{s\in\R^k}$ and $(N_t,p')_{t\in\R^k}$.
	 Then for every integer $\ell>0$ there exists an algebraic
	biholomorphism $(B^\ell,\varphi^\ell)$ between
	 $(M_s,p)_{s\in\R^k}$ and $(N_t,p')_{t\in\R^k}$ which agrees with $(B,\varphi)$
	up to order $\ell$ at $(p,0)$.
\end{Thm}

Under the stronger hypothesis 
that $(M_0,p)$ is "finitely nondegenerate",  Theorem~\ref{thm:main2app} was proved by 
Baouendi, Rothschild, and Zaitsev (see \cite{BRZ2}). However, the weakening of the nondegeneracy 
assumption makes  it impossible to use their methods. 

Let us restate our results in a more geometric fashion. For this, we use the following
notation.
We say that two germs 
of real-algebraic CR submanifolds $(M,p)$ and $(M',p')$ of $\C^N$ are biholomorphically equivalent, and 
write $(M,p)\sim_h (M',p')$ if there exists a germ of a biholomorphism $H\colon (\CN,p) \to (\CN,p')$ 
and a neighbourhood $U$ of $p$ in $\C^N$ such that $H(M\cap U) \subset M'$ (we shall abbreviate this by writing $H(M)\subset M'$). We say that $(M,p)$ and $(M',p')$ are 
algebraically equivalent, and write $(M,p)\sim_a (M',p')$, if there
exists such a biholomorphism which is furthermore algebraic. 

Let us recall that if $M\subset \CN$ is a real-algebraic CR submanifold, for every $q\in M$ there exists a unique germ of a real-algebraic submanifold $\mathcal{W}_q$ through $q$ with 
the property that every (small) piecewise differentiable curve starting at $q$, whose tangent vectors are 
in the complex tangent space, has its image contained in $\mathcal{W}_q$ (see \cite{BER96}). $\mathcal{W}_q$ is referred to 
as the {\em local CR orbit} of $q$. We shall say that $(M,p)$ has {\em constant orbit dimension} if  $\dim \mathcal{W}_q$ is 
constant for $q$ close by $p$. The geometric counterpart
of Theorem~\ref{thm:main2} can now be stated as follows. 

\begin{Thm}\label{thm:main3} Let $(M,p)$ be a germ of a holomorphically nondegenerate real-algebraic CR 
	submanifold, which is in addition of constant orbit dimension. Assume $(M',p')$ is a germ of a real-algebraic 
	submanifold of $\CN$ for which $(M,p)\sim_h (M',p')$. Then $(M,p)\sim_a (M',p')$.
\end{Thm}

Also Theorem~\ref{thm:main3} is a consequence of an approximation theorem, which can be stated as follows. 

\begin{Thm}\label{t:main4} Let $(M,p)\subset \C^N$ be a germ of a real-algebraic CR submanifold which is 
	holomorphically nondegenerate and of constant orbit dimension. 
	Then for every real-algebraic CR submanifold $M'\subset \C^N$ and every
	positive integer $\ell$, if $h\colon (\C^N,p)\to \C^N$ is the germ of a biholomorphic map satisfying $h(M)\subset M'$,
	there exists an algebraic biholomorphism $h^\ell \colon (\C^N,p)\to \C^N$ satisfying $h^\ell(M)\subset M'$ which agrees with $h$ up to order $\ell$ at $p$. 
\end{Thm} 

Let us briefly recall why 
a CR submanifold of constant orbit dimension is a deformation of its 
CR orbits: for a germ of a real-algebraic CR manifold $(M,p)$ which is of constant orbit dimension, there
exists an integer $k\in \{0,\ldots,N\}$ and a real-algebraic submersion $S \colon (M,p) \to (\R^k,0)$ such 
that $S^{-1} (S(q)) = \mathcal{W}_q =: M_{S(q)}$ for all $q\in M$ near $p$.  
The level sets of $S$ therefore foliate $M$ by {\em minimal} real-algebraic CR submanifolds, 
and thus $M$ is an algebraic deformation of $M_0$ (see \cite{BRZ2} or Lemma~\ref{l:coordconstantorbit}). In addition, a local biholomorphism sending two such real-algebraic CR submanifolds is also a biholomorphism of the associated deformations, since CR orbits of the source manifold are mapped to CR orbits of the target manifold (see e.g.\ \cite{BERbook}). On the other hand, given an algebraic deformation 
$(M_s)_{s\in\R^k}$ of a minimal, real-algebraic CR submanifold $(M_0,0)\subset (\Cn,0) $, 
the manifold $(M,0)\subset (\C^{n+k},0)$ defined by $ M = \{ (z,w) \in (\C^{n+k},0)\colon z\in M_{\re w}, \im w = 0\}$ is 
a real-algebraic CR submanifold of constant orbit dimension. Hence the statements given by Theorem~\ref{thm:main2app} and Theorem~\ref{t:main4} are equivalent.

If one removes the assumption about 
holomorphic nondegeneracy in Theorem~\ref{t:main4}, we can still show that the conclusion of the theorem holds if we assume a certain ``mild'' form of holomorphic degeneracy. We shall say that a point $p$ in a real-analytic CR submanifold $M\subset \C^N$ is a {\em regular point of the holomorphic foliation on $M$} if there exists an integer $k\in \{0,\ldots,N-1\}$ and a holomorphically nondegenerate real-analytic CR submanifold $\widehat M \subset \C^{N-k}$ such that $(M,p)\sim_h (\widehat M \times \C^k,0)$. (This notion is motivated by the structure of the holomorphic foliation arising in holomorphically degenerate CR submanifolds, that is discussed in detail in section~\ref{s:foliation}). We have the following result: 

%This "mild" of holomorphic degeneracy 

%comes from the structure of the holomorphic foliation existingWhen $(M,p)$ is a germ of a holomorphically degenerate real-algebraic CR submanifold of $\C^N$, there exists a germ of a nontrivial holomorphic vector field tangent to $M$ near $p$. The existence of such vector fields gives rise to a holomorphic foliation on $M$ (possibly with singularities) that is described in detail in section \ref{s:foliation}.

\begin{Thm}\label{t:main2holfol} Let $M\subset \C^N$ be a connected real-algebraic CR submanifold. Then the following holds:

\begin{enumerate}
\item [(i)] the set of all points $p\in M$ such that $p$ is a regular point of the holomorphic foliation on $M$ and $(M,p)$ is of constant orbit dimension is the complement of a closed proper real-algebraic subvariety $\Sigma_M$ of $M$.
\item[(ii)] For every point $p\in M\setminus \Sigma_M$, for every real-algebraic CR submanifold $M'\subset \C^N$ and for every positive integer $\ell$, if  $h\colon (M,p)\to M'$ is the germ of a biholomorphic map satisfying $h(M)\subset M'$,
	there exists an algebraic biholomorphism $h^\ell \colon (M,p)\to M'$ satisfying $h^\ell (M)\subset M'$ which agrees with $h$ up to order $\ell$ at $p$. 
\end{enumerate}
\end{Thm}

%The two properties in Theorem~\ref{t:main2holfol} each hold generically on a real-algebraic CR submanifold $M$ of $\CN$; i.e. 
%t%here exists a proper real-algebraic subvariety $\Sigma_M^1\subset M$ such that at each point 
%$p\in M\setminus\Sigma_M^1$, $(M,p)$ is of constant orbit dimension, and there exists a real-algebraic subvariety 
%$\Sigma_M^2\subset M$ such that  each point 
%$p\in M\setminus\Sigma_M^2$ is a regular point of the holomorphic foliation on $M$. 

% (see Remark~\ref{r:comparebrz} (i) for details). As mentioned before, Theorem~\ref{t:main} provides a new positive 
% answer to the question of
% holomorphic versus algebraic equivalence for the case of nowhere minimal real-algebraic submanifolds, which 
% is left open in view of \cite{BMR} already cited above. As a consequence of Theorem~\ref{t:main}, we
% would like to mention the following result:
% 
% \begin{Cor}\label{c:cor} Let $M\subset \C^N$ be a connected holomorphically nondegenerate real-algebraic CR submanifold and
% $\Sigma_M^1$ be the above defined proper real-algebraic subvariety of $M$. Then for every point $p\in M \setminus
% \Sigma_M^1$, for every real-algebraic CR submanifold $M'\subset \C^N$, and for every point $p'\in M'$, 
% biholomorphic equivalence of the germs $(M,p)$ and $(M',p')$ implies their algebraic equivalence. \end{Cor}
% 

The proof of Theorem~\ref{t:main4}  
is based on a careful study of some algebraic properties of local biholomorphic mappings
sending (nowhere minimal) real-algebraic CR submanifolds into each other. Given $(M,p)$ as in Theorem~\ref{t:main4}, we may assume without loss of generality that $p=0$ and that $M$ is generic in $\C^N$. As already explained, $M$ can be viewed as a deformation of its CR orbits and therefore we identify $M$ with a deformation $(M_s)_{s\in \R^k}$ of a certain minimal holomorphically nondegenerate real-algebraic generic submanifold $M_0\subset \C_z^{N-k}$ passing through $0$. The first step of the proof is to determine the dependence of a given biholomorphic map  $h=h(z,u)$ with respect to the "parameter" $u\in \C^k$. Indeed, by the algebraicity theorem proved in \cite{BER96}, one already knows  that for every fixed $s\in \R^k$ small enough, the map $z\mapsto h(z,s)$ is algebraic. We prove that there exists an integer $m_0$ such the holomorphic map $(z,u)\mapsto h(z,u)$ depends algebraically on  the functions $z,u$ and $\6H(u):=((\partial^{\gamma}h)(0,u);|\gamma|\leq m_0)$. This kind of parameter dependence result can not be obtained by using the techniques of \cite{BER96, BRZ2}. It is obtained as a combination of some previous results of the second author \cite{Mir98, MirCAG} and a key algebraic property, Proposition~\ref{p:keyalg}, proved in section~\ref{s:algprop}.

%Given $(M,p)$, $(M',p')$, and $h$ as in 
%Theorem~\ref{t:main4}, we
%associate to the mapping $h$ a ring of convergent power series that are  ``partially algebraic'' (for 
%an exact statement, see
%section~\ref{ss:keyproperty}). We then prove  that each component of the
%mapping $h$ must necessarily be algebraic over the quotient field of this ring
%(see Propositions \ref{p:keyalg} and \ref{p:mircag}). In order to do so, 
%we need several intermediate results due to the second author 
%\cite{Mir98, MirCAG}, and we use the Segre set technique introduced in \cite{BER96}. 

Once Proposition~\ref{p:mircag} is combined with Proposition~\ref{p:keyalg}, one obtains a system of (holomorphic) polynomial equations over $\C_z^{N-k}\times \C^k_u$ satisfied by the mappings $h$ and $\6H$. At this point one could apply Artin's approximation theorem  \cite{A69} to 
approximate $(h,\6H)$ by algebraic mappings. However, the sequence of algebraic mappings approximating $h$ need not send $M$ to $M'$. This problem is overcome by constructing another polynomial system of real-algebraic equations over $\R^k$ fulfilled by the mapping $\6H$. An application of a more refined version of Artin's approximation theorem due to Popescu \cite{P} to the new polynomial system coupled with the first one (more precisely to its restriction to $\R^N)$ provides the desired conclusion.

%comes from the structure of the holomorphic foliation existingWhen $(M,p)$ is a germ of a holomorphically degenerate real-algebraic CR submanifold of $\C^N$, there exists a germ of a nontrivial holomorphic vector field tangent to $M$ near $p$. The existence of such vector fields gives rise to a holomorphic foliation on $M$ (possibly with singularities) that is described in detail in section \ref{s:foliation}.

To derive Theorem~\ref{t:main2holfol}, one needs to study holomorphically degenerate real-analytic CR manifolds and understand the structure of the holomorphic foliation (with singularities) arising from the existence of holomorphic vector fields tangent to them. This is done in section~\ref{s:foliation} where we, in addition, show that if the manifolds are algebraic, the holomorphic foliation is algebraic.

The paper is organized as follows. Section~\ref{s:algprop} is
devoted to the proof of an algebraic property for certain holomorphic mappings whose restriction on a nowhere minimal
real-algebraic CR manifold satisfy some special type of polynomial identity. In section~\ref{s:proofmt}, we prove the main
approximation result of the paper, Theorem~\ref{t:technic}; Theorems~\ref{thm:main2},\ref{thm:main2app},\ref{thm:main3},
and \ref{t:main4} are
direct consequences of this result. In Section~\ref{s:foliation}, we recall several basic facts about the structure of the
holomorphic foliation (with singularities) on a (holomorphically degenerate) real-analytic CR submanifold $M\subset \C^N$. If $M$ is 
real-algebraic and connected, we show that this foliation is algebraic; in particular, 
the singular locus of this foliation, which coincides with the complement of the set of regular points of the foliation, is a closed proper real-algebraic subvariety of $M$. We then deduce in section~\ref{s:last} Theorem~\ref{t:main2holfol} from Theorem~\ref{t:technic} and the results of section~\ref{s:foliation}. The basic background on CR analysis needed throughout the paper may be found e.g.\ in the books
\cite{Bog, BERbook}.

\subsection*{Acknowledgments:} The authors would like to thank an anonymous referee for 
his helpful insights and comments, which helped us to essentially 
improve an earlier version of this
paper.

\section{Preliminaries and Notation}

We start by recalling some basic facts and introducing our notation. 
\subsection{Algebraic functions and mappings} Throughout the paper,  $\C\{x\}$ denotes the ring of convergent
power series with complex coefficients in the variables $x=(x_1,\ldots,x_r)$, $r\geq 1$. The ring $\C\{x\}$ can
be identified with the ring of germs of holomorphic functions at $0$ in $\C^r$, and we 
shall do so freely. Given any convergent power series
$\eta=\eta (x)\in \C\{x\}$, we denote by $\bar{\eta}=\bar{\eta}(x)$ the convergent power series obtained from $\eta$
by taking complex conjugates of its coefficients.

An element $f\in \C\{x\}$ is  {\em algebraic} (or  {\em Nash}) if it satisfies a nontrivial polynomial identity with
polynomial coefficients, i.e. if it is algebraic over the subring $\C[x]\subset \C\{x\}$. 
We  denote the subring of $\C \{x\}$ of all algebraic power
series by $\6N \{x\}$. To be completely explicit, this means that 
$f\in\6N \{ x\}$ if $f \in \C\{x\}$ and there
exist polynomials $p_j \in \C[x]$ for $j=0,\dots,m$ with
$p_m \neq 0$ such that \[ \sum_{j=0}^m p_j (x) f(x)^j = 0. \]
Given nonnegative integers $k$ and $s$, a germ of a holomorphic map $(\C^k,0)\to \C^s$ is 
algebraic if all of its components are algebraic.
We also have to consider  the ring of germs at $0$ in $\R^r$ of (complex-valued)
real-algebraic functions. This ring will be denoted by $\6N^{\R}\{x\}$ and coincides with the ring of germs at $0$ in
$\R^r$ of (complex-valued) real-analytic functions whose complexification belongs to $\6N\{x\}$.

\subsection{CR orbits, normal coordinates and iterated Segre mappings}\label{ss:crorbit}

Let $M\subset \C^N$ be a real-analytic CR submanifold and $T^{0,1}M$ its CR bundle. We denote by 
${\mathcal G}_M$ the Lie algebra  generated by the sections of $T^{0,1}M$ and its conjugate $T^{1,0}M$, 
and by 
${\mathcal G}_M (p) \subset \C T_p M$ the space of the evaluations at $p$ of these sections. 
By a theorem of Nagano (see e.g.\ \cite{BERbook, BCH08}), for every point $p\in M$, there is a well-defined unique germ at $p$ of a real-analytic submanifold $\mathcal{W}_p$ satisfying $\C T_q\mathcal{W}_p={\mathcal G}_M(q)$ for all $q\in V_p$. This unique submanifold is necessarily CR and called the {\em local CR orbit} of $M$ at $p$. Note that this definition coincides with the one given in the introduction using piecewise differentiable curves running in complex tangential directions (see e.g.\ \cite{BERbook}). Since $M$ is real-analytic (resp.\  real-algebraic), it is easy to see that if $M$ is connected, the dimension of the local CR orbits is constant except possibly on a closed proper real-analytic (resp.\ real-algebraic) subvariety $\Sigma_M^1$ of $M$ (see e.g.\ \cite{BRZ1, BRZ2}). If ${\rm dim}\, \mathcal{W}_p={\rm dim}\, M$ for some point 
$p\in M$, we say that $M$ is  {\em minimal} (or also of {\em finite type}) at $p$.

Let $M$ be a generic real-algebraic submanifold of $\C^N$ of CR dimension $n$ and codimension $d$. 
We recall that a point $p\in M$ is of constant orbit dimension if it has an
open neighbhorhood $U$ in $M$ such that $\dim \mathcal{W}_q$ is constant for $q\in U$. In this case, 
one may describe the obtained algebraic
foliation by CR orbits in terms of normal coordinates as follows. First recall that
coordinates $Z = (z,\eta)\in\Cn\times\Cd$ are {\em normal coordinates} for $(M,p)$ if $p=0$ and there exists a map 
$\Theta(z,\chi,\eta)$ defined in a neighbourhood of $0 \in \C^{2n +d}$ and satisfying the {\em normality conditions}
\begin{equation}\label{e:normalnormal} 
	\Theta(z,0 ,\eta) = \Theta(0,\chi,\eta) = \eta,  
	\end{equation}
	such that  $M$ is given by 
$\eta = \Theta(z,\bar z , \bar \eta)$. The map $\Theta$ also satisfies the {\em reality condition}
\begin{equation}
	\label{e:realnormal} \Theta (z,\chi,\bar \Theta(\chi,z,\eta)) = \eta.
\end{equation} 

\begin{Lem}\label{l:coordconstantorbit} $($\cite[Proposition 3.4]{BRZ1}\, {\rm and}\, \cite[Lemma 3.4.1]{BER96}$)$ 
	Let $(M,p)\subset \C^N$ be a germ of a  generic real-algebraic submanifold which is of constant orbit dimension at $p$, 
	and denote by $c\in \{0,\ldots, d\}$  the codimension of the CR orbits close by $p$ in $M$. 
	Then there exist normal algebraic coordinates 
	$Z=(z,\eta)\in \C^n \times \C^d$, $Z=(z,w,u)\in \C^n \times \C^{d-c}\times \C^c$, 
	such that $M$ is given near the origin by an equation of the form
\begin{equation}\label{e:conormal}
\eta=(w,u)=\Theta (z,\bar z,\bar \eta):=(Q(z,\bar z, \bar w,\bar u),\bar u), 
\end{equation}
where $Q(z,\chi,\tau,u)$ is a $\C^{d-c}$-valued algebraic map near $0\in \C^{n+N}$. Furthermore, there exist 
neighborhoods $U,V$ of the origin in $\R^c$ and $\C^{N-c}$ respectively such that for every $u\in U$, the real-algebraic 
submanifold given by 
\begin{equation}\label{e:Mu}
M_u:=\{(z,w)\in V: w=Q(z,\bar z, \bar w,u)\}
\end{equation}
is generic and minimal at $0$.
\end{Lem}

A real-analytic submanifold
$M$, given by $\eta = \Theta(z,\bar z , \bar \eta)$, can be ``complexified''; 
its  {\em complexification}, which 
we 
denote by $\6M$, is the germ at $0$ of the complex submanifold of $\C^{2N}$ given by 
\begin{equation}\label{e:complexification}
\{(Z,\zeta)\in (\C^N\times \C^N,0):\sigma =\bar \Theta (\chi,z,\eta)\},
\end{equation}
where $Z=(z,\eta)\in \C^n \times \C^d$ and $\zeta=(\chi,\sigma)\in \C^n \times \C^d$.

In the whole paper, we always choose coordinates for our germ $(M,p)$ according to  
Lemma~\ref{l:coordconstantorbit}. 
% 
% For the proof of Theorem~\ref{t:main}, 
We also  need to recall the {\em iterated Segre mappings} $v_j$ attached to $(M,p)$ 
(see e.g.\ \cite{BRZ1}), which we define as follows. First, to simplify notation,
for any positive integer $j$, 
we  denote by $t^j$ a variable lying in $\C^{n}$ and also introduce the variable 
$t^{[j]}:=(t^1,\ldots,t^j)\in \C^{nj}$. For $j=1$, we set 
$v_1(t^1,u):=(t^1,\Theta(t^1,0,(0,u)))$ for $t^1\in \C^n$ and $u\in \C^c$ sufficiently close to $0$; 
for $j>1$ we inductively define 
$v_j\colon (\C^{nj}\times \C^c,0)\to \C^N$   as follows:
\begin{equation}\label{e:iteratedmaps}
v_j(t^{[j]},u):=(t^j,\Theta (t^j, \bar v_{j-1}(t^{[j-1]},u))).
\end{equation}
Define $v_0(u):=(0,u)\in \C^N$,  so that \eqref{e:iteratedmaps} also holds for $j=1$. 
From the construction, each iterated Segre mapping $v_j$ defines  an algebraic map in 
a neighbhorhood of $0$ in $\C^{nj} \times \C^c$.  From \eqref{e:realnormal} one obtains the  identities
\begin{equation}\label{e:identities}
v_j(0,u)=(0,u),\quad v_{j+2}(t^{[j+2]},u)|_{t^{j+2}=t^{j}}=v_{j}(t^{[j]},u),\quad j\geq 0.
\end{equation}
Furthermore, for every $j\geq 1$, the germ at $0$ of the holomorphic map $(v_{j},\bar v_{j-1})$ 
takes its values in $\6M$, where $\6M$ is the complexification of $M$ as defined by \eqref{e:complexification}.

\section{A key algebraic property for holomorphic mappings}\label{s:algprop}

\subsection{Statement of the algebraicity property}\label{ss:keyproperty}
In this section, we assume that $f\colon (\C^N,0)\to (\C^{N'},0)$,  where $N\geq 2$, $N'\geq 1$, is 
a germ of a holomorphic mapping and that 
$(M,0)\subset \C^N$ is a germ of a  real-algebraic  generic submanifold of CR dimension $n$ and codimension $d$. We also assume that $0$ is a point of constant orbit dimension in $M$, 
and that coordinates for $(M,0)$ have been chosen according to Lemma~\ref{l:coordconstantorbit};  $\6M$ denotes the complexification of $M$ defined  in  
Section~\ref{ss:crorbit}. 
For every integer $\ell$, we define the subring $\6R^f_\ell$ of $\C\{Z,\zeta\}$ to consist of
power series which can be written in the form 
$$C(Z,\zeta,(\partial^{\gamma}\bar{f}(\zeta))_{|\gamma|\leq \ell}),$$ for
some $C\in \6N \{Z,\zeta\}[(\Lambda_\gamma)_{|\gamma|\leq \ell}]$, where for each 
$\gamma \in \N^N$, $\Lambda_\gamma \in
\C^{N'}$. Let $\6I_{\6M}$ be the ideal of $\C\{Z,\zeta\}$ of those convergent power series that vanish on $\6M$. 
Denote by $\C \{\6M\}$ the coordinate ring of $\6M$, i.e. quotient ring
$\C\{Z,\zeta\}/{\6I}_{\6M}$ and let $\pi_{\6M}\colon \C\{Z,\zeta\}\to \C \{\6M\}$ be the natural projection. 
We identify a convergent power series $J(Z)$ or $L(\zeta)$ 
with its image in $\C\{\6M\}$ via $\pi_{\6M}$ ($\pi_{\6M}$ is injective on $\C\{Z\}$ since $M$ is assumed to be generic). 
Define
$\6S^f_\ell:=\pi_{\6M}(\6R^f_\ell)$. We now say 
that $f$ satisfies assumption $(\clubsuit)$, if there exists a positive integer $\ell_0$ such that 
each component of the power series mapping $f$ (considered
as an element of the ring $\C\{\6M\}$) is algebraic over the quotient field of $\6S^f_{\ell_0}$.
To be more concrete, this means that
\begin{enumerate}
% \item[($\bigstar$)] The local CR orbits of $M$ are all of the same dimension near $0$ and associated normal coordinates $Z=(z,w,u)$ have been chosen as in Lemma~\ref{l:coordconstantorbit}.
\item[($\clubsuit$)] there exists an integer $\ell_0$, integers $k_1,\dots , k_{N'}$,
%for every integer $s\in \{1,\ldots,N'\}$ an integer $k_s$, 
and a family of algebraic power series 
$\delta_{r,s}=\delta_{r,s}\left(Z,\zeta,(\Lambda_{\gamma})_{|\gamma|\leq
\ell_0}\right)\in \6N \{Z,\zeta\}[(\Lambda_{\gamma})_{|\gamma|\leq
\ell_0}]$, $r\in \{1,\ldots,k_s\}$, $s = 1,\dots, N'$, such that
\begin{equation}\label{e:firstbasic0}
\sum_{r=0}^{k_s}\delta_{r,s}\left(Z,\zeta,(\partial^{\gamma}\bar{f}(\zeta))_{|\gamma|\leq
\ell_0}\right) \left(f_s(Z)\right)^r =0
\end{equation}
holds for $(Z,\zeta)\in \6M$ near the origin, with
$\delta_{k_s,s}\left(Z,\zeta,(\partial^{\gamma}\bar{f}(\zeta))_{|\gamma|\leq
\ell_0}\right)\not \equiv 0$ on ${\6M}$. 
\end{enumerate}
Our goal in this section is to prove that if $(\clubsuit)$ is fulfilled, then
the mapping $f$ is necessarily algebraic over a certain field of convergent power series that are 
partially algebraic. To be more precise, we need to define the following rings. 

\begin{Def}\label{d:ah}Let $(z,w,u)\in \C^n \times \C^{d-c}\times \C^c$ be our fixed chosen normal coordinates. For every nonnegative integer $\ell$, let $\6A^f_\ell$ be the subring of $\C \{z,w,u\}$ consisting of those convergent power series $T=T(z,w,u)$ which can be written in the form $A(z,w,u,(\partial^{|\alpha|} f(0,u))_{|\alpha|\leq \ell})$ for some $A(z,w,u,(\Lambda_\alpha)_{|\alpha|\leq \ell})\in \6N \{z,w,u\}[(\Lambda_\alpha)_{|\alpha|\leq \ell}]$. We  furthermore consider the ring $\6A^f$ defined by 
\begin{equation}\label{e:ah0}
\6A^f:=\bigcup_{j=0}^\infty \6A_j^f.
\end{equation}
\end{Def}
Note that the rings $\6A^f_\ell$ depend on the (fixed) choice of normal coordinates.
We can now state the main result of this section.

\begin{Pro}\label{p:keyalg}
Let $(M,0)\subset \C^N$ be a real-algebraic generic submanifold which 
is of constant orbit dimension at $0$, and $f\colon (\C^N,0)\to (\C^{N'},0)$ be a germ of a holomorphic mapping. 
If $f$ satisfies $(\clubsuit)$, then each component of the mapping $f$ is algebraic over the quotient field of $\6A^f$.
\end{Pro}

The rest of this section is devoted to the proof of Proposition~\ref{p:keyalg}.

\subsection{Algebraic dependence over quotient fields of certain rings}\label{ss:rings}

The first step in the proof of Proposition~\ref{p:keyalg} is given by the following lemma. Its proof is very similar to the proof of \cite[Proposition~5.2]{MirCAG}.

\begin{Lem}\label{l:lemderivative}
Let $M,f$ be as in Proposition~{\rm \ref{p:keyalg}}. Then, for every multiindex $\mu \in  \N^N$, each component of the mapping $\partial^\mu f$ is algebraic over the quotient field of $\6S^f_{\ell_0+|\mu|}$.
\end{Lem}

\begin{proof} For $\mu=0$, this is the content of assumption $(\clubsuit)$. Differentiating \eqref{e:firstbasic0} once, each first order derivative of each component of the mapping $f$ is algebraic over the quotient field of the subring of $\C\{\6M\}$ generated by $f$ and $\6S^f_{\ell_0+1}$ (this follows from the chain rule; see \cite[Proposition~5.2]{MirCAG} for exactly similar arguments). Since (each component of) the mapping $f$ is already algebraic over the quotient field of $\6S^f_{\ell_0}\subset \6S^f_{\ell_0+1}$, this proves the proposition for all multiindices $\mu$ of length one. The conclusion for multiindices of arbitrary length follows in the same way by induction. 
\end{proof}

\subsection{Iterated Segre mappings and associated rings}

Our next step is to use the iterated Segre mappings as introduced in Section~\ref{ss:crorbit}. For this, we need to introduce even more subrings. 
Let $j$ be a nonnegative integer. Recall that $t^j$ denotes a variable lying in $\C^n$ and $t^{[j]}$ stands for
$(t^1,\ldots,t^j)$. For every such $j$ and every integer $\ell$, we let $\bar{\6B}^f_{j,\ell}$ be the subring of
$\C\{t^{[j+1]},u\}$ consisting of the convergent power series of the form
$$K(t^{[j+1]},u,((\partial^{|\alpha|} \bar f)\circ \bar v_j(t^{[j]},u))_{|\alpha|\leq \ell})$$
for some $K\in \6N \{t^{[j+1]},u\}[(\Lambda_\alpha)_{|\alpha|\leq \ell}]$. Similarly to before, we also set 
\begin{equation}\label{e:bj}
\bar{\6B}^f_j:=\cup_{\ell =0}^\infty \bar{\6B}^f_{j,\ell}.
\end{equation}

For every integer $\ell$, define also the ring $\bar{\6D}_{j,\ell}^f$ 
to be the subring of $\C\{t^{[j+1]},u\}$ consisting of those power series of the form
$$A(t^{[j+1]},u,((\partial^{|\alpha|} \bar f)(0,u))_{|\alpha|\leq \ell})$$
for some $A\in \6N \{t^{[j+1]},u\}[(\Lambda_\alpha)_{|\alpha|\leq \ell}]$ 
and we also set $\bar{\6D}^f_j:=\cup_{\ell =0}^\infty \bar{\6D}^f_{j,\ell}$. 

Analogously, define the ring ${\6D}_{j,\ell}^f$ as the subring of 
$\C\{t^{[j+1]},u\}$ consisting of the power series of the form
$$R(t^{[j+1]},u,((\partial^{|\alpha|}  f)(0,u))_{|\alpha|\leq \ell})$$
for some $R\in \6N \{t^{[j+1]},u\}[(\Lambda_\alpha)_{|\alpha|\leq \ell}]$ 
and  set $\6D^f_j:=\cup_{\ell =0}^\infty \6D^f_{j,\ell}$.

Our next step is to prove the following result.

\begin{Lem}\label{l:iterate1}
Let $M,f$ be as in Proposition~{\rm \ref{p:keyalg}}. With the above notation, for every multiindex $\mu \in \N^N$ and every integer $j$, each component of the mapping $(\partial^\mu f)\circ v_{j+1}$ is algebraic over the quotient field of $ \bar{\6B}^f_j$.
\end{Lem}

In order to prove Lemma~\ref{l:iterate1}, we need the following lemma which
is a kind of ``step-down''  procedure for 
algebraicity over the rings $\bar{\6B}^f_{j+1}$; it 
can be seen as an adaptation of \cite[Lemma~5.4]{MirCAG} to our situation.

\begin{Lem}\label{l:induction} Assume that we are in the setting of Lemma~{\rm \ref{l:iterate1}}, and let $j$ be a positive integer. Assume that $g\colon (\C^N,0)\to \C$ is a holomorphic function such that
	$g\circ v_{j+2}$ is algebraic over the quotient field of $\bar{\6B}^f_{j+1}$. Then $g\circ v_{j}$ 
	is algebraic over the quotient field of $\bar {\6B}^f_{j-1}$.
\end{Lem}

In what follows, to shorten the notation, we write $v_j$ instead of $v_j(t^{[j]},u)$.

\begin{proof}[Proof of Lemma~{\rm \ref{l:induction}}]  
	By assumption, there exist  positive integers $e$ and $\ell$, and 
	a family of power series $\delta_{r}\in \6N \{t^{[j+2]},u\}[(\Lambda_\alpha)_{|\alpha|\leq \ell}]$,
	 $r=0,\ldots,e$, such that
\begin{equation}\label{e:firsteq}
\sum_{r=0}^e \delta_r\left(t^{[j+2]},u,((\partial^{|\alpha|} \bar f)\circ \bar v_{j+1})_{|\alpha|\leq \ell}\right) (g\circ v_{j+2})^r=0,
\end{equation}
with 
$\delta_e (t^{[j+2]},u,((\partial^{|\alpha|} \bar f)\circ \bar v_{j+1})_{|\alpha|\leq \ell})\not \equiv 0$. Let
$\nu\in \N^n$ be  a multiindex of minimal length with respect to the 
property that 
there exists $r\in \{1,\ldots,e\}$ satisfying 
$$\frac{\partial^{|\nu|}}{\partial (t^{j+2})^\nu}\left[ \delta_r\left(t^{[j+2]},u,((\partial^{|\alpha|} \bar f)\circ \bar v_{j+1})_{|\alpha|\leq \ell}\right)\right]\biggr|_{t^{j+2}=t^j}\not \equiv 0.$$
Applying $\frac{\partial^{|\nu|}}{\partial (t^{j+2})^\nu}$ to \eqref{e:firsteq}, evaluating for $t^{j+2}=t^j$, and using the second identity in \eqref{e:identities}, 
we obtain that
\begin{equation}\label{e:secondeq}
\sum_{r=0}^e \frac{\partial^{|\nu|}}{\partial (t_{j+2})^\nu} \left[\delta_r\left(t^{[j+2]},u,((\partial^{|\alpha|} \bar f)\circ \bar v_{j+1})_{|\alpha|\leq \ell}\right)\right]\biggr|_{t^{j+2}=t^j} (g\circ v_{j})^r=0.
\end{equation}
We write 
\begin{multline}\label{e:thirdeq}
	\widetilde \delta_r (t^{[j+1]},u,((\partial^{|\alpha|} \bar f)\circ \bar v_{j+1})_{|\alpha|\leq \ell}))\\:=
	\frac{\partial^{|\nu|}}{\partial (t_{j+2})^\nu} \left[\delta_r\left(t^{[j+2]},u,((\partial^{|\alpha|} \bar f)\circ \bar v_{j+1})_{|\alpha|\leq \ell}\right)\right]\biggr|_{t^{j+2}=t^j},
\end{multline}
and observe that by our choice of $\nu$, 
there exists $\tilde r \in \{1,\ldots,e\}$ 
such that 
\[ \widetilde \delta_r \left(t^{[j+1]},u,((\partial^{|\alpha|} \bar f)\circ \bar v_{j+1})_{|\alpha|\leq \ell})\right)
\not\equiv 0.
\]
Hence, if we
choose $\beta \in \N^n$ such that
$$\frac{\partial^{|\beta|}}{\partial (t^{j+1})^\beta}\left[ \widetilde \delta_r 
\left(t^{[j+1]},u,((\partial^{|\alpha|} \bar f)\circ \bar v_{j+1})_{|\alpha|\leq \ell})\right) \right]_{t^{j+1}=t^{j-1}}\not \equiv 0,$$
and for $0\leq r\leq e$ write
$$\widehat \delta_r:=\frac{\partial^{|\beta|}}{\partial (t^{j+1})^\beta}\left[ \widetilde \delta_r 
\left(t^{[j+1]},u,((\partial^{|\alpha|} \bar f)\circ \bar v_{j+1})_{|\alpha|\leq \ell})\right) \right]\biggr|_{t^{j+1}=t^{j-1}},$$
we see that each $\widehat \delta_r \in \bar{\6B}^f_{j-1,\ell+|\beta|}$ and that the nontrivial relation
$$\sum_{r=0}^e \widehat \delta_r (g\circ v_{j})^r=0$$
provides the desired result. The proof of the lemma is complete.
\end{proof}

\begin{proof}[Proof of Lemma~{\rm \ref{l:iterate1}}] By Lemma~\ref{l:lemderivative}, for every multiindex $\mu$ and every integer $s\in \{1,\ldots,N'\}$, there exists an integer $e(\mu,s)$, a family of power series 
$\Delta^\mu_{r,s}\left(Z,\zeta,(\Lambda_{\gamma})_{|\gamma|\leq
\ell_0+|\mu|}\right)\in \6N \{Z,\zeta\}[(\Lambda_{\gamma})_{|\gamma|\leq
\ell_0+|\mu|}]$, $r\in \{0,\ldots,e(\mu,s)\}$, such that for all $(Z,\zeta)\in \6M$ (near the origin)
\begin{equation}\label{e:newid}
\sum_{r=0}^{e(\mu,s)}\Delta^\mu_{r,s}\left(Z,\zeta,(\partial^{\gamma}\bar{f}(\zeta))_{|\gamma|\leq
\ell_0+|\mu|}\right)(\partial^\mu f_s (Z))^r=0
\end{equation}
with \begin{equation}\label{e:display}
\Delta^\mu_{e(\mu,s),s}\left(Z,\zeta,(\partial^{\gamma}\bar{f}(\zeta))_{|\gamma|\leq
\ell_0+|\mu|}\right)\not \equiv 0\  {\rm on}\  \6M, 
\end{equation}
and where $\ell_0$ is given by condition $(\clubsuit)$. For every integer $j$, the algebraic map $(v_{j+1},\bar v_j)$ takes its values in $\6M$ and therefore, from \eqref{e:newid}, we have the following identities
\begin{equation}\label{e:newideas}
\sum_{r=0}^{e(\mu,s)}\Delta^\mu_{r,s}\left(v_{j+1},\bar v_j,((\partial^{\gamma}\bar{f})\circ \bar v_j)_{|\gamma|\leq
\ell_0+|\mu|}\right) \left((\partial^\mu f_s)\circ v_{j+1}\right)^r =0.
\end{equation}
In order to see that  \eqref{e:newideas} implies that  $(\partial^\mu f)\circ v_{j+1}$ 
is algebraic over the quotient
field of $\bar{\6B}^f_j$ for all $j\geq 0$, we note that by Lemma~\ref{l:induction} 
it is enough to check this for all $j\geq d+1$. Now, if $j\geq d+1$, we claim that the map
\[(\C^{n(j+1)}\times \C^c,0)\ni (t^{[j+1]},u)\mapsto (v_{j+1}(t^{[j+1]},u),\bar 
v_j(t^{[j]},u)), \]
which takes values in $\6M$, has generic rank $N+n = \dim \6M$; from this we conclude that 
for these $j$, by \eqref{e:display}, the relation \eqref{e:newideas} is nontrivial.

We now turn to the proof of this last claim, which is a consequence of the 
finite
type criterion in \cite{BER96}. 
Let $M_0$ be the real-algebraic generic submanifold given by \eqref{e:Mu}. 
Since $M_0$ is of finite type, by the 
finite
type criterion in \cite{BER96}, the map $(\C^{n(d+1)},0)\ni t^{[d+1]}\mapsto v_{d+1}(t^{[d+1]},0)$ is of generic rank
$n+d-c$, where $d,c$ are as in Lemma~\ref{l:coordconstantorbit}. Therefore, for any $j\geq d+1$, the mapping
$(\C^{nj}\times \C^c,0)\ni (t^{[j]},u)\mapsto v_{j}(t^{[j]},u)$ is of generic rank $n+d-c+c=N$. 
Hence for all such $j$'s,
the generic rank of the mapping $(\C^{n(j+1)}\times \C^c,0)\ni (t^{[j+1]},u)\mapsto (v_{j+1}(t^{[j+1]},u),\bar 
v_j(t^{[j]},u))$ is
equal to $N+n={\rm dim}\, \6M$.  The proof of Lemma~\ref{l:iterate1} is complete. \end{proof}

\subsection{Proof of Proposition~\ref{p:keyalg}} An application of Lemma~\ref{l:iterate1} for $j=0$ yields that for all multiindices $\mu \in \N^N$, each component of the mapping $\partial^\mu f \circ v_1$ is algebraic over the quotient field of $\bar{\6B}_0^f=\bar{\6D}^f_0$.
% , which coincides with the quotient field of $\bar{\6D}^f_0$ since =\bar{\6B}_0^f$. 
Therefore, the conjugate mapping $\partial^\mu \bar f \circ \bar v_1$ is algebraic over the quotient field of 
$\6D^f_0$; since for all $\nu\in \N^N$, each component of the mapping $\partial^\nu  f \circ  v_2$ 
is algebraic over the quotient field of $\bar{\6B}_1^f$
by Lemma~\ref{l:iterate1},  the  transitivity of being algebraic implies
that each component of the mapping $\partial^\nu  f \circ  v_2$ is algebraic over the quotient field of $\6D^f_1$. Proceeding inductively, we see that for every multiindex $\beta \in \N^N$ and every even integer $j$, the components of the map $\partial^\beta f \circ v_j$ are all algebraic over the quotient field of $\6D^f_{j-1}$ and for every odd integer $j$ the same algebraicity property holds over the quotient field of $\bar{\6D}^f_{j-1}$. Hence for $j=2(d+1)$,  there exists  positive integers $b,e$, convergent power series $\Psi_{\nu}\in \6N \{t^{[2(d+1)]},u\}[(\Lambda_\gamma)_{|\gamma|\leq b}]$, $\nu \in \{0,\ldots,e\}$, such that a nontrivial relation of the following form
\begin{equation}\label{e:almostthere}
\sum_{\nu=0}^{e}\Psi_{\nu}\left(t^{[2(d+1)]},u,((\partial^{\gamma}{f})(0,u))_{|\gamma|\leq
b})\right)(f\circ v_{2(d+1)}(t^{[2(d+1)]},u))^\nu=0
\end{equation}
holds for all $(t^{[2(d+1)]},u)$ sufficiently close to the origin.

In order to see that \eqref{e:almostthere} implies that each component of $f$ is algebraic over 
the quotient field of $\6A^f$, we need to invert the map $v_{2(d+1)}$. 
By the minimality criterion given in \cite{BER96, BERbook}, there exists points arbitrarily close to the origin in $\C^{2(d+1)n}$ such that $t^{[2(d+1)]}\mapsto v_{2(d+1)}(t^{[2(d+1)]},0)$ is of rank $N-c$ at those points with image the origin in $\C^{N}$. Pick a point $T^0\in \C^{2(d+1)n}$ with the above property and such that \eqref{e:almostthere} holds near $(T^0,0)\in \C^{2(d+1)n} \times \C^c$. From the rank theorem, there exists an algebraic mapping $\theta\colon (\C^N,0)\to (\C^{N-c},T^0)$ such that $v_{2(d+1)}(\theta (z,w,u),u)=(z,w,u)$. Composing \eqref{e:almostthere} with the obtained left inverse for $v_{2(d+1)}$, we obtain an identity of the following form
\begin{equation}\label{e:hereweare}
\sum_{\nu=0}^{e}\Psi_{\nu}\left(\theta (z,w,u),u,((\partial^{\gamma}{f})(0,u))_{|\gamma|\leq
b})\right)(f(z,w,u))^\nu=0,
\end{equation}
for $(z,w,u)\in \C^N$ sufficiently close to the origin. Furthermore, it is not difficult to see that it is possible to choose the mapping $\theta$ so that the obtained relation \eqref{e:hereweare} is still nontrivial. Hence \eqref{e:hereweare} shows that $f$ is algebraic over the quotient field of $\6A^f$. This finishes the proof of Proposition~\ref{p:keyalg}.

\section{Proof of Theorems~\ref{thm:main2},\ref{thm:main2app},\ref{thm:main3},
and \ref{t:main4}}\label{s:proofmt}

The goal of this section is to prove the following approximation theorem; Theorem~\ref{t:main4} (and therefore also Theorems~\ref{thm:main2},\ref{thm:main2app},\ref{thm:main3}) are
a direct consequence of this more general theorem, from which we will also deduce
Theorem~\ref{t:main2holfol} in Section~\ref{s:last}.
% Recall that if $M$ is a connected real-algebraic CR submanifold in $\C^N$,
% we denote, as in the introduction, by $\Sigma_M^1$ the proper real-algebraic subvariety of $M$ defined in
% Section~\ref{ss:crorbit}, i.e. the set of points whose local CR orbits fail to have maximal 
% dimension.

\begin{Thm}\label{t:technic} Let $(M,p)\subset \C^N$ be a germ of a 
	 holomorphically nondegenerate real-algebraic generic
submanifold, which is of constant orbit dimension at $p$. 
Assume that $M'\subset \C^N$ is a real-algebraic generic submanifold, and that 
$h\colon (\C^N,p) \to \C^N$ 
is  a germ of a holomorphic map  with $h(M)\subset M'$ and ${\rm Jac}\,h\not \equiv 0$. Then for every positive integer $\ell$, 
there exists a germ of 
an algebraic mapping $h^\ell \colon (\C^N,p)\to \C^N$ satisfying $h^\ell (M)\subset M'$ and that agrees with $h$ at 
$p$ up to order $\ell$. \end{Thm}

In order to prove Theorem~\ref{t:technic}, % we let
% $M,M'\subset \C^N$  be two fixed real-algebraic generic
% submanifolds of the same dimension through distinguished points $p\in M\setminus \Sigma_M^1$ and $p'\in M'$.
we
assume without loss of generality that $p=0$, $p'=h(p)=0$, and that normal coordinates $Z=(z,w,u)$ have been chosen for $M$
near $0$ as in Lemma~\ref{l:coordconstantorbit}. 
We let $h\colon (\C^N,0)\to (\C^N,0)$ be a germ of a holomorphic map of
generic full rank as in the statement of Theorem~\ref{t:technic}. 
We will need the following result from \cite{MirCAG}
that follows from an inspection of the proof of \cite[Proposition~4.6]{MirCAG}.

\begin{Pro}\label{p:mircag}
Let $M$, $M'$, and $h$ be as above.  Then the mapping $h$ satisfies assumption $(\clubsuit)$ given 
in Section {\rm \ref{s:algprop}}.
\end{Pro}

Combining Proposition~\ref{p:mircag} and Proposition~\ref{p:keyalg} we get that each component of the mapping $h$ is 
algebraic over the quotient field of the ring $\6A^h$ defined in Definition~\ref{d:ah}, i.e. 
there exists an integer $m_0$ and, 
for every integer $s\in \{1,\ldots,N\}$, an integer $k_s$, and
a family of convergent power series 
$P_{r,s}(Z,(\Lambda_{\alpha})_{|\alpha|\leq m_0})\in \6N \{Z\}[(\Lambda_{\alpha})_{|\alpha|\leq m_0}]$, $0\leq r\leq k_s$, such that 
\begin{equation}\label{e:finaldependence}
\sum_{r=0}^{k_s}P_{r,s}(Z,(\partial^{|\alpha|}h(0,u))_{|\alpha|\leq m_0}) (h_s(Z))^r=0,
\end{equation}
with 
\begin{equation}\label{e:nonzero}
P_{k_s,s}(Z,(\partial^{|\alpha|}h(0,u))_{|\alpha|\leq m_0})\not \equiv 0.
\end{equation}

For ease of notation, we write 
$\6H(u):=(\partial^{|\alpha|}h(0,u))_{|\alpha|\leq m_0}$. Note that \eqref{e:finaldependence} simply means
that the convergent power series mapping $(h,\6H)$ satisfies a certain polynomial system with algebraic coefficients;
however, solutions of this system need not necessarily give rise to holomorphic mappings sending $(M,0)$ to $(M',0)$. The
goal of the subsequent paragraphs is to build up an additional system of polynomial equations with real-algebraic coefficients
satisfied by the mapping $\6H$ that will allow us to approximate the mapping $h$ by algebraic mappings in the
Krull topology.

\subsection{Construction of an appropriate system of polynomial equations} In what follows, we use the following notation
to denote coordinates on jet spaces: with $m_0$ as given above, we write $\Lambda=(\Lambda_\alpha)_{|\alpha|\leq m_0}$
where each $\Lambda_\alpha \in \C^N$ for $\alpha \in \N^N$. Similarly, we write
$\Gamma=(\Gamma_\alpha)_{|\alpha|\leq m_0}$ with $\Gamma_\alpha \in \C^N$. Let us also recall that coordinates in the
source space $\C^N$ split as  $Z=(z,w,u)\in \C^n\times \C^{d-c}\times \C^c$ and similarly for $\zeta \in \C^N$, where we
write $\zeta=(\chi,\tau,v)\in \C^n\times \C^{d-c}\times \C^c$.

The following lemma will allow us 
to construct a suitable system of polynomial equations satisfied by the power series $\6H(u)$. 
Its proof is  analogous to that of \cite[Lemma~6.2]{MMZ3}.

\begin{Lem}\label{l:crucial}
For every real-valued polynomial $r'=r'(Z',\bar Z')$, $Z'\in \C^N$, there exists a nontrivial polynomial 
\[ \6K=\6K (Z,\zeta,\Lambda,\Gamma;X) =\sum_{\nu=0}^\delta B_\nu (Z,\zeta,\Lambda,\Gamma)X^\nu
\in \6N \{Z,\zeta\} [\Lambda,\Gamma][X]\]
with the following properties:
\begin{enumerate}
\item[(i)]   $B_{\delta}(Z,\zeta,\6H(u),\bar{\6H}(v))\not \equiv 0$ for $(Z,\zeta)\in \6M$ near the origin.
\item[(ii)] for every pair of convergent power series mappings 
$S(u)=(S_\alpha (u))_{|\alpha|\leq m_0}$ and $F(Z)=(F_1(Z),\ldots,F_N(Z))$ which  satisfies 
\begin{equation}\label{e:ouf}
\sum_{r=0}^{k_s}P_{r,s}(Z,S(u)) (F_s(Z))^r=0
\end{equation}
for every 
$s\in \{1,\ldots,N\}$, 
with 
\begin{equation}\label{e:nonzerouf}
P_{k_s,s}(Z,S(u))\not \equiv 0,
\end{equation}
we have that 
\begin{equation}\label{e:crucialid}
\6K(Z,\zeta,S(u),\bar S(v);r'(F(Z),\bar F(\zeta)))\equiv 0,
\end{equation}
for $(Z,\zeta)\in \C^{2N}$ sufficiently close to the origin.
\end{enumerate}
\end{Lem}

\begin{proof} 
For every $s\in  \{1,\ldots,N\}$, denote by $T_s$ and $Y_s$ new indeterminates and consider 
\begin{equation}\label{e:new}
\6Q_s (Z,\Lambda;T_s):=\sum_{r=0}^{k_s}P_{r,s}(Z,\Lambda)T_s^r,\quad {\6R}_s(\zeta,\Gamma;Y_s):=\sum_{r=0}^{k_s}\bar{P}_{r,s}(\zeta,\Gamma)Y_s^r.
\end{equation}
$\6Q_s$ and ${\6R}_s$ are polynomials with algebraic coefficients.
Let $\Delta$ be a sufficiently small polydisc in $\C^N$ centered at the origin so that each $P_{r,s}$ is holomorphic in $\Delta \times \C^\kappa$, where $\kappa:= N {\rm Card}\, \{\alpha \in \N^N:|\alpha|\leq m_0\}$. For each $s=1,\ldots,N$, denote by $L_s$ (resp.\ $\bar L_s$) the zero set of $P_{k_s,s}$ (resp.\ $\bar{P}_{k_s,s}$) in $\Delta \times \C^\kappa$ and let $E:=\cup_{s=1}^N(L_s\cup \bar{L}_s)$. For every $(Z,\Lambda)\in (\Delta \times \C^\kappa)\setminus E$, every $(\zeta,\Gamma)\in (\Delta \times \C^\kappa)\setminus E$ and every $s\in \{1,\ldots,N\}$, denote by $\sigma^{(s)}_1(Z,\Lambda),\ldots,\sigma^{(s)}_{k_s}(Z,\Lambda)$ the $k_s$ roots of the polynomial $\6Q_s$ (viewed as a polynomial in $T_s$). Hence, for every $(\zeta,\Gamma)\in (\Delta \times \C^\kappa)\setminus E$ and every $s\in \{1,\ldots,N\}$,
$\bar{\sigma}^{(s)}_1(\zeta,\Gamma),\ldots,\bar{\sigma}^{(s)}_{k_s}(\zeta,\Gamma)$ are 
the $k_s$ roots of the polynomial $\6R_s$.
% Note that for every $(Z,\Lambda)\in (\Delta \times \C^\kappa)\setminus E$, we have $(\bar Z,\bar{\Lambda})\in (\Delta \times \C^\kappa)\setminus E$ and that $\sigma^{(s)}_j (Z,\Lambda)=\overline{\sigma_j^{(s)}(\bar Z,\bar \Lambda)}$ for all $j=1,\ldots,k_s$ and $s\in \{1,\ldots,N\}$, which justifies the abuse of notation made here. 

As in \cite[Lemma~6.2]{MMZ3}, consider for $(Z,\Lambda)\in (\Delta \times \C^\kappa)\setminus E$ and $(\zeta,\Gamma)\in (\Delta \times \C^\kappa)\setminus E$ the following polynomial in $X$
\begin{multline}\label{e:thepolynomial}
W(Z,\zeta,\Lambda, \Gamma;X):=\\
\prod_{\substack{ 1\leq l_j \leq k_j\\1\leq n_j \leq N_j}}
% \prod_{l_1=1}^{k_1}\ldots \prod_{l_{N}=1}^{k_{N}}
% \prod_{n_1=1}^{k_1}\ldots \prod_{n_{N}=1}^{k_{N}} 
\Big(X-
r'\big(\sigma^{(1)}_{n_1}(Z,\Lambda),\ldots,\sigma^{(N)}_{n_{N}}(Z,\Lambda),\1 \sigma^{(1)}_{l_1}(\zeta,\Gamma),\ldots, \1
\sigma^{(N)}_{l_{N}}(\zeta,\Gamma)\big)\Big).
\end{multline}

By Newton's theorem on symmetric polynomials, it follows that \eqref{e:thepolynomial} can be rewritten as
\begin{equation}\label{e:anrfwf}
W(Z,\zeta, \Lambda,\Gamma;X)=X^\delta+\sum_{\nu <\delta}A_\nu(Z,\Lambda,\zeta,\Gamma)X^\nu,
\end{equation}
where $\delta:=\prod_{j=1}^Nk_i^2$ and where $A_\nu$ is of the following form
\begin{equation}\label{e:coef}
A_\nu(Z,\zeta,\Lambda,\Gamma):=C_\nu \left(  \left( \frac{P_{r,s}(Z,\Lambda)}{P_{k_s,s}(Z,\Lambda)},\frac{\bar{P}_{r,s}(\zeta,\Gamma)}{\bar{P}_{k_s,s}(\zeta,\Gamma)}
\right)_{\substack{1\leq r\leq k_s \\ 1\leq s\leq N}}\right),
\end{equation}
where each $C_\nu$ is a polynomial of its arguments depending only on $r'$. 
Let 
\begin{equation}\label{e:I}
I(Z,\Lambda,\Lambda,\Gamma):=\prod_{s=1}^NP_{k_s,s}(Z,\Lambda)P_{k_s,s}(\zeta,\Gamma).
\end{equation}
For a suitable integer $\varpi$,
$\6K(Z,\zeta,\Lambda,\Gamma;X):=I^\varpi\cdot W\in \6N\{Z,\zeta\}[\Lambda,\Gamma][X]$. 
We claim that the obtained polynomial $\6K$ satisfies all desired properties. 
The construction of the polynomial $\6K$ implies (ii). 
To prove statement (i), we note that the term $B_\delta$ is a sufficiently high power of the function $I$ defined in 
\eqref{e:I}. Hence if $B_\delta(Z,\zeta,\6H(u),\bar{\6H}(v))\equiv 0$ for $(Z,\zeta)\in \6M$ near $0$, then 
there exists $s\in \{1,\ldots,N\}$ such that either $P_{k_s,s}(Z,\6H(u))\equiv 0$ for $Z\in \C^N$ close to $0$, or 
$\bar{P}_{k_s,s}(\zeta,\bar{\6H}(v))\equiv 0$ for $\zeta \in \C^N$ close to $0$, which is 
not the case by  
\eqref{e:nonzero}. The proof of Lemma~\ref{l:crucial} is complete.
\end{proof}

Next, since $M'$ is a real-algebraic generic submanifold of codimension $d$ through the origin, we may choose  $d$ real-valued polynomials $(r_1',\ldots,r_d')$ such that $M'$ is given near the origin by the zero set of these $d$ polynomials.
Applying Lemma~\ref{l:crucial} for each polynomial $r_j'$, $j=0,\ldots,d$, we obtain a corresponding polynomial 
\[\6K^j(Z,\zeta,\Lambda,\Gamma;X):=\sum_{\nu=0}^{\delta_j}B_\nu^jX^\nu\]
with each $B_\nu^j(Z,\zeta,\Lambda,\Gamma)\in \6N\{Z,\zeta\}[\Lambda,\Gamma][X]$, $1\leq \nu\leq \delta_j$.

For every $j\in \{1,\ldots,d\}$, define 
\begin{equation}\label{e:nuj}
\nu^j_0=\min \{ \nu \in \{0,\ldots,\delta_j\} \colon B^j_{\nu}(Z,\zeta,\6H(u),\bar{\6H} (v))\not \equiv 0\  {\rm on}\  \6M\ {\rm near}\ 0 \}.
\end{equation}
Since each polynomial $\6K^j$ satisfies  conclusion (i) of Lemma~\ref{l:crucial}, 
we know that such a $\nu^j_0$ exists (it can be shown  that each $\nu^j_0>0$ but this is not needed in what follows).

Using Lemma~\ref{l:coordconstantorbit}, we  choose  a real-algebraic parametrization $(\R^{2N-d-c}\times \R^c,0)\ni (y,u)\mapsto \varphi (y,u)\in \C^N$ of $M$ near $0$ which satisfies 
that for each fixed $u\in \R^c$ sufficiently close to the origin, the mapping $(\R^{2N-d-c},0)\ni y\mapsto 
\varphi (y,u)$ parametrizes the manifold $M_u$ as defined in \eqref{e:Mu} near the origin. Hence for all $(y,u)\in 
\R^{2N-d-c}\times \R^c$ sufficiently close to the origin, we have
\begin{equation}\label{e:formalsatisfied}
B^j_\nu (\varphi (y,u),\overline{\varphi (y,u)},\6H(u),\overline{\6H(u)})=0,\quad j=0,\ldots,d,\quad   \nu=0,\ldots\nu^j_0-1.
\end{equation}
For every $j\in \{0,\ldots,d\}$, for every $\nu<\nu^j_0$, and for every multiindex $\gamma \in \N^{2N-d-c}$, define
\begin{equation}\label{e:firstsystem}
\Theta_{\gamma,\nu}^j(u,\Lambda,\overline{\Lambda}):=
\frac{\partial^{|\gamma|}}{\partial y^\gamma}
\left[B^j_\nu \left(\varphi (y,u),\overline{\varphi (y,u)},\Lambda,\overline{\Lambda}\right)\right]
\biggr |_{y=0}\in \6N^{\R}\{u\}[\Lambda,\overline{\Lambda}].
\end{equation}
Since the ring $\6N^{\R}\{u\}[\Lambda,\overline{\Lambda}]$ is noetherian, there exists an integer $n_0$ such that the ideal generated by the $\Theta_{\gamma,\nu}^j(u,\Lambda,\overline{\Lambda})$ for $\gamma \in  \N^{2N-d-c}$, $\nu<\nu^j_0$, $j\in \{0,\ldots,d\}$ coincides with that generated by the $\Theta_{\gamma,\nu}^j(u,\Lambda,\overline{\Lambda})$ for $|\gamma|\leq n_0$, $\nu<\nu^j_0$, $j\in \{0,\ldots,d\}$. 
We observe that by  the construction of the mappings $\Theta_{\gamma,\nu}^j$ the following Lemma holds.

\begin{Lem}\label{l:tobeused}
For $\gamma \in  \N^{2N-d-c}$, $\nu<\nu^j_0$, $j\in \{0,\ldots,d\}$, let $\Theta_{\gamma,\nu}^j$ be defined as above. Then if $(\R^c,0)\ni u\mapsto S(u)\in \C^\kappa$ is the germ of a real-analytic mapping satisfying
$$ \Theta_{\gamma,\nu}^j(u,S(u),\overline{S(u)})=0$$
for $u\in \R^c$ close to the origin and  $|\gamma|\leq n_0$, $\nu<\nu^j_0$, $j\in \{0,\ldots,d\}$, then $S$ satisfies
$$ B^j_\nu (Z,\overline{Z},S(u),\overline{S(u)})=0,\quad \nu=0,\ldots\nu^j_0-1,\quad   j=0,\ldots,d,$$
for all $Z=(z,w,u)\in M$ sufficiently close to the origin.
\end{Lem}

We can now define the crucial system of complex-valued real-analytic equations 
near the origin in $\R_x^n\times \R_t^{d-c}\times \R_u^c$ to carry out the approximation:

\begin{equation}\label{e:system}
\begin{cases}
\begin{aligned}
\sum_{r=0}^{k_s}P_{r,s}(x,t,u,\Lambda)\,  T_s^r&=0, &&s=1,\ldots,N\\
\Theta_{\gamma,\nu}^j(u,\Lambda,\overline{\Lambda})&=0, && |\gamma|\leq n_0,\ \nu<\nu^j_0,\ j\in \{0,\ldots,d\}.
\end{aligned}
\end{cases}
\end{equation}
This system is obviously polynomial in $T=(T_1,\ldots,T_N)$ and $(\Lambda,\overline{\Lambda})$ with coefficients that are real-algebraic functions in $\R^N$. Furthermore, using \eqref{e:finaldependence}, \eqref{e:formalsatisfied} and \eqref{e:firstsystem}, we know that \[\begin{aligned}
	T&=(T_1,\ldots,T_N)=(h_1(x,t,u),\ldots,h_N(x,t,u))=h(x,t,u),\\ \Lambda&=\6H(u)=(\partial^{|\alpha|}h(0,u))_{|\alpha|\leq m_0}
\end{aligned}
\] is a complex-valued real-analytic solution of the system \eqref{e:system}.

\subsection{Appropriate solutions of the system and end of the proof of Theorem~\ref{t:technic}}
At this point, the next natural step would be to apply an approximation theorem due to Artin \cite{A69} providing, 
for every integer $\ell$,  complex-valued real-algebraic solutions $h^\ell(x,t,u)$ and $W^\ell(x,t,u)$ of the 
system \eqref{e:system}, that agree up to order $\ell$ at the origin with the mappings $h(x,t,u)$ and $\6H(u)$ respectively. 
This strategy is not sufficient 
because in order to show that the complexification of the obtained mappings
$h^\ell$ sends a neighbhorhood of $0$ in $M$ to $M'$, we need  the mappings $W^\ell$  to be 
independent of $(x,t)$ (as is the case for the original solution). The question whether this kind  of approximation
is possible actually dates back to Artin's paper \cite{A69}; it became known as the ``subring condition''. 
A positive answer in the algebraic case has been provided  by Popescu \cite{P}, and this theorem allows us to get solution 
mappings $W^\ell$ 
which depend only on $u$. We shall only state a version of Popescu's theorem needed for the purpose of this paper.

\begin{Thm}\label{t:popescu} $(${\rm Popescu  \cite{P}}$)$ Let $(\omega,\xi)\in \R^k\times \R^q$, $k,q\geq 1$, 
	and a polynomial mapping  $\Phi(X,Y)=(\Phi_1(X,Y),\ldots,\Phi_r(X,Y))$
	with $\Phi_j(X,Y) \in \6N^{\R}\{\omega,\xi\}[X,Y]$ for $j=1,\dots,r$, 
	$X=(X_1,\ldots,X_p)$, $Y=(Y_1,\ldots,Y_m)$. Suppose that there exists formal power series mappings 
% 	$\phi (\omega) \in (\R [[\omega]])^r$, $\psi (\omega,\xi)\in (\R[[\omega,\xi]])^m$ safisfying $\Phi (\phi,\psi)=0$. 
% 	Then for every integer $\ell$, there exists $\phi^\ell \in (\6N^{\R}\{\omega\})^r$, $\psi^\ell \in (\6N^{\R}\{\omega,\xi\})^m$
% satisfying  $\Phi (\phi^\ell,\psi^\ell)=0$ and that agree at $0$ up to order $\ell$ with $\phi$ and $\psi$ respectively.
	$x (\omega) \in (\R [[\omega]])^p$, $y (\omega,\xi)\in (\R[[\omega,\xi]])^m$ safisfying $\Phi (x,y )=0$. 
	Then for every integer $\ell$, there exists $x^\ell \in (\6N^{\R}\{\omega\})^p$, $y^\ell \in (\6N^{\R}\{\omega,\xi\})^m$
satisfying  $\Phi (x^\ell,y^\ell)=0$ such that $x^\ell$ and $y^\ell$ agree at $0$ up to order $\ell$ with $x$ and $y$ respectively.
\end{Thm}

\begin{Rem} The fact that the system of polynomial equations has algebraic coefficients in the above result is of fundamental importance. Indeed, the
analogous statement for polynomial systems with analytic coefficients does not hold in general, as a well-known example due to Gabrielov
\cite{Gab} shows. \end{Rem}

Applying Theorem~\ref{t:popescu} to the  real equations associated to the 
system of polynomial equations given by
\eqref{e:system}, we obtain, for every
positive integer $\ell$, a real-analytic
$\C^N$-valued algebraic mapping
$h^\ell=h^\ell(x,t,u)$ and a
real-analytic $\C^\kappa$-valued
algebraic mapping $W^\ell=W^\ell(u)$,
both defined in a neighborhood of $0\in
\R^N$ (depending on $\ell$) and agreeing
with $h(x,t,u)$ and $\6H(u)$ up to order
$\ell$ at $0$ respectively. We
complexify the mappings $h^\ell$ and
$W^\ell$ without changing the notation.
Since each mapping $h^\ell=h^\ell(Z)$ is
algebraic and agrees with the mapping
$h=h(Z)$ up to order $\ell$ at $0$, the
proof of Theorem~\ref{t:technic} is completed 
by the following Lemma.

\begin{Lem}\label{l:final}
In the above setting and with the above notation, for $\ell$ sufficiently large, the holomorphic map $h^\ell$ sends a neighborhood of $0$ in $M$ to $M'$.
\end{Lem}

\begin{proof}[Proof of Lemma~{\rm \ref{l:final}}] Recall first that $(r_1',\ldots,r_d')$ are $d$ real-valued polynomials such that $M$ is given near the origin by the zero set of these $d$ polynomials. Suppose, by contradiction, that the conclusion of the lemma does not hold. Considering a subsequence if necessary, we may assume, without loss of generality, that for every integer $\ell$, \begin{equation}\label{e:assumption}
r_1'(h^\ell(Z),\overline{h^\ell}(\zeta))\not \equiv 0,\ (Z,\zeta)\in \6M,\ {\rm near}\ 0.
\end{equation}
Let 
\begin{equation}\label{e:add}
\6K^1(Z,\zeta,\Lambda,\Gamma;X)=\sum_{\nu=0}^{\delta_1}B_\nu^1(Z,\zeta,\Lambda,\Gamma)X^\nu 
\end{equation}
be the polynomial given by Lemma~\ref{l:crucial} associated to the real polynomial $r_1'$. Since $(h^\ell,W^\ell)$ satisfies the system of equations  \eqref{e:system} for every integer $\ell$, we have for all $Z\in \C^N$ sufficiently close to the origin
$$\sum_{r=0}^{k_s}P_{r,s}(Z,W^\ell(u))\,  (h^\ell_s(Z))^r=0,\ s=1,\ldots,N.$$
Since $\6K^1$ satisfies property (ii) of Lemma~\ref{l:crucial}, we therefore obtain
\begin{equation}\label{e:vienna}
\6K^1(Z,\zeta,W^\ell(u),\overline{W^\ell}(v);r_1'(h^\ell(Z),\overline{h^\ell}(\zeta)))\equiv 0,
\end{equation}
for $(Z,\zeta)\in \C^{2N}$ close to the origin. In what follows, we shall restrict \eqref{e:vienna} to the complexification $\6M$. Since $W^\ell$ satisfies the second equation of the system \eqref{e:system},
Lemma~\ref{l:tobeused} implies that $ B^1_\nu (Z,\overline{Z},W^\ell(u),\overline{W^\ell(u)})=0$, for $0\leq \nu<\nu^1_0$, and for all $Z=(z,w,u)\in M$ sufficiently close to the origin. Equivalently, we have for $(Z,\zeta)\in \6M$ near $0$
\begin{equation}\label{e:sick}
B^1_\nu (Z,\zeta,W^\ell(u),\overline{W^\ell}(v))=0,\quad  0\leq \nu<\nu^1_0.
\end{equation}
Hence \eqref{e:add}, \eqref{e:vienna} and \eqref{e:sick} imply that
\begin{equation}\label{e:thisisit}
\sum_{\nu=\nu^1_0}^{\delta_1}B^1_\nu (Z,\zeta,W^\ell(u),\overline{W^\ell}(v))\bigg( r_1'(h^\ell(Z),\overline{h^\ell}(\zeta))\bigg)^\nu=0,\quad (Z,\zeta)\in \6M,\ {\rm near}\ 0.
\end{equation}
Using \eqref{e:assumption}, \eqref{e:thisisit}  yields that for $(Z,\zeta)\in \6M$ sufficiently close to $0$,
\begin{multline}\label{e:end?}
-B^1_{\nu^1_0}(Z,\zeta,W^\ell(u),\overline{W^\ell}(v))\\= \sum_{\nu=\nu^1_0+1}^{\delta_1}B^1_\nu (Z,\zeta,W^\ell(Z),\overline{W^\ell}(\zeta))\bigg( r_1'(h^\ell(Z),\overline{h^\ell}(\zeta))\bigg)^{\nu-\nu^1_0}.
\end{multline}
Since $h$ sends a neighbhorhood of $M$ to $M'$ and since, for every integer $\ell$, the mapping $h^\ell$ agrees with $h$ up to order $\ell$ at $0$, it follows that $r_1'(h^\ell(Z),\overline{h^\ell}(\zeta))|_{\6M}$ vanishes at least to order $\ell$ at the origin. Hence from \eqref{e:end?}, the same property holds for $B^1_{\nu^1_0}(Z,\zeta,W^\ell(u),\overline{W^\ell}(v))|_{\6M}$. But also 
$W^\ell$ agrees with $\6H$ up to order $\ell$ at the origin,
so for every integer $\ell$, the germ at $0$ of the holomorphic function $B^1_{\nu^1_0}(Z,\zeta,\6H(u),\overline{\6H}(v))|_{\6M}$ vanishes at least up to order $\ell$. 
As a consequence,
$$B^1_{\nu^1_0}(Z,\zeta,\6H(u),\overline{\6H}(v))\equiv 0,\quad \text{for }(Z,\zeta)\in \6M  \text{ near } 0,$$
which contradicts the definition of $\nu^1_0$ by \eqref{e:nuj}. 
This completes the proof of Lemma~\ref{l:final} and therefore the proof of Theorem~\ref{t:technic}.
\end{proof}

\section{Holomorphic foliations in  CR manifolds}\label{s:foliation}

The main goal of this section is to study the holomorphic foliation (with singularities) arising from the existence of holomorphic vector fields tangent to a real-analytic (resp.\ real-algebraic) CR submanifold in complex space. The main result of this section is given by Proposition~\ref{p:algstraight} and shows that the holomorphic foliation of a real-algebraic holomorphically degenerate CR submanifold of $\C^N$ is in fact algebraic. The results of this section 
will be used to derive Theorem~\ref{t:main2holfol} from Theorem~\ref{t:technic} in section
\ref{s:last}.

%\subsection{Holomorphic degeneracy for real-analytic and real-algebraic CR manifolds}\label{ss:algsub} In this section, we
%shall establish a few properties of holomorphic foliations on (holomorphically degenerate) real-analytic and real-algebraic
%CR submanifolds of $\C^N$. These properties seem to be known in the real-analytic case and follow essentially from
%\cite{BR95}.

\subsection{The real-analytic case}\label{sss:choose} Let $M\subset \C^N$ be a real-analytic generic
submanifold of CR dimension $n$ and codimension $d$ so that $N=n+d\geq 2$. In what follows, for a point $p\in
\C^N$, we denote by $\4O_p$ the ring of germs of holomorphic functions at $p$ and by $\4M_p$ its quotient field
of meromorphic functions at $p$; these are the stalks of the sheaf of holomorphic (resp. meromorphic) functions,
which we will accordingly denote by $\4O$ and $\4M$, respectively.

A holomorphic (resp.\ meromorphic) vector field on an open set 
$U\subset \CN$ is a holomorphic (resp. meromorphic) 
section of $T^{(1,0)} \CN$ over $U$, i.e. an expression of the form
\[ X = \sum_{j=1}^N a_j (Z) \dopt{}{Z_j}, \quad a_j \in \4O (U) \, \text{(} \4M (U) \text{ respectively)}.\]
Again the stalks of these sheaves  at $p\in\CN$, i.e. expressions of the form
\[ X = \sum_{j=1}^N a_j (Z) \dopt{}{Z_j}, \quad a_j \in \4O_p \, \text{(} \4M_p \text{ respectively)},\]
will be referred to as germs of holomorphic (respectively meromorphic) vector fields.  
We identify the sheaf of holomorphic (resp. meromorphic) vector fields
with $\4O^N$ ($\4M^N$, respectively), and the  germs of holomorphic (resp.\ meromorphic)
vector fields at a point $p\in \C^N$ with $(\4O_p)^N$ (resp.\ $(\4M_p)^N$).

For $p\in M$, let $\4T_p$ (resp.\ $\4S_p$) be the set of all germs at $p$ of holomorphic (resp.\ meromorphic)
vector fields that are tangent to $M$. $\4T_p$ is an $\4O_p$-submodule of the free module $(\4O_p)^N$; $\4S_p$
carries the structure of a finite dimensional vector space over $\4M_p$. As in \cite{BR95}, for all $p\in M$,
define \begin{equation}\label{e:lambda} \lambda_M (p):={\rm dim}_{\4M_p}\, \4S_p \ \in \{0,\ldots,N-1\}.
\end{equation} We recall from \cite[Section 4]{BR95} a known characterization of $\lambda_M (p)$ through local coordinates. For an arbitrary point $p\in M$ we may choose
(see e.g.\ \cite{BERbook}) local holomorphic coordinates $Z=(z,\eta)\in \C^n \times \C^d$, vanishing at $p$,
such that $M$ is given near the origin by a vector-valued equation of the form
\begin{equation}\label{e:generdefinequa} \eta=\Theta (z,\bar z, \bar \eta), \end{equation} where $\Theta=\Theta
(z,\chi,\sigma)$ is a $\C^d$-valued holomorphic map near $0$ satisfying the reality condition
\begin{equation}\label{e:reality} 
	\Theta (z,\chi,\bar \Theta (\chi,z,\eta))\equiv \eta,
\end{equation} 
	and
\begin{equation} 
	\label{e:normality} \Theta (z,0,\eta) = \Theta (0,\chi, \eta) = \eta; 
\end{equation} 
	such
coordinates are commonly referred to as {\em normal coordinates} for $M$ at $p$. If $M$ is furthermore assumed
to be real-algebraic, we may choose these local holomorphic coordinates as well as the mapping
$\Theta$ to be algebraic.

Recall also that such a choice of normal coordinates can be made for points $q\in M$ nearby $p$ in such a way
that the mapping $\Theta$ depends real-analytically on $q$ (and real-algebraically if $M$ is real-algebraic, see
\cite{BERbook}). We expand the mapping $\bar{\Theta}$ into a Taylor series as follows:
\begin{equation}\label{e:expansion} \bar{\Theta} (\chi,z,\eta):=\sum_{\beta \in \N^n}\bar \Theta_\beta
(z,\eta)\chi^\beta. \end{equation}

In what follows, we keep the above notation and choice of coordinates for a given fixed point $p\in M$. We also write the coordinates $Z=(Z_1,\ldots,Z_N)$. We need to
recall  the following slight generalization of a known criterion (see e.g.\ \cite[Lemma~4.5]{BR95}).

\begin{Lem}\label{l:stanton} In the above setting,  the followings holds :
\begin{enumerate}
\item[(i)] a germ of a holomorphic  vector field $X=\sum_{j=1}^Na_j(Z)\frac{\partial}{\partial Z_j}$ with $a_j\in \4O_0$  is tangent to $M$ if and only if
\begin{equation}\label{e:stantoncriterion}
\sum_{j=1}^Na_j(Z)\frac{\partial \bar{\Theta}_\beta}{\partial Z_j}(Z)\equiv 0,\quad \text{ for all } \beta \in \N^n;
\end{equation}
the same holds for germs of meromorphic vector fields with $\4O_0$ replaced by 
$\4M_0$.
\item[(ii)] there exists a neighborhood $U$ of $0$ in $\C^N$ such that for every point $q\in M\cap U$, a germ of a holomorphic vector field $Y=\sum_{j=1}^Nb_j(Z)\frac{\partial}{\partial Z_j}$ with $b_j\in \4O_q$ is tangent to $M$ $($near $q)$ if and only if 
\begin{equation}\label{e:stantoncriterionbis}
\sum_{j=1}^Nb_j(Z)\frac{\partial \bar{\Theta}_\beta}{\partial Z_j}(Z)\equiv 0,\quad  \text{ for all } \beta \in \N^n,
\end{equation}
for all $Z\in \C^N$ sufficiently close to $q$.
\end{enumerate}
\end{Lem}

\begin{proof} It is enough to prove (ii). We fix a polydisc $U_1\subset \C^n$ and $U_2\subset \C^d$ both containing the origin such that 
$\bar \Theta$ is holomorphic in $U_1\times U_1\times U_2$. Let $U:=U_1\times U_2$. Given $q=(z_q,\eta_q)\in M\cap U$, it is easy to see that a germ at $q$ of a holomorphic vector field $Y$ is tangent to $M$ near $q$ if and only  $Y(\bar \Theta(\chi, Z))=0$ for all $Z$ in some connected neighborhood $U_q\subset \C^N$ of $q$ and for all $\chi$ in some neighbhorhood  $\widetilde U_1\subset \C^n$ of $\bar z_q$. Since
$$Y(\bar \Theta(\chi, Z))=\sum_{j=1}^Nb_j(Z)\frac{\partial \bar \Theta}{\partial Z_j}(\chi, Z),$$
the map $Y(\bar \Theta (\chi, Z))$ is in fact holomorphic in $U_1\times U_q$ and vanishes on $\widetilde U_1\times U_q$. Hence $Y(\bar \Theta (\chi, Z))$ vanishes identically in $U_1\times U_q$. From this,  the desired conclusion \eqref{e:stantoncriterionbis} follows.
\end{proof}

We have the following result (see \cite{BR95}).

\begin{Lem}\label{l:substitute} Let $M\subset \C^N$ be a real-analytic generic submanifold with $N\geq 2$, $p\in
M$, and $(z,\eta)$ normal coordinates for $M$ near $p$. Then the following identity holds:
\begin{equation}\label{e:kappa} \lambda_M(p)+r_M(p)=N, \end{equation} where $r_M(p)$ is the generic rank of the
holomorphic map $(\C^N,0) \ni (z,\eta)\mapsto (\bar{\Theta}_\beta(z,\eta))_{\beta \in \N^n}$ defined by {\rm
\eqref{e:expansion}}. \end{Lem} \begin{proof} Choose an integer $\ell_0$ large enough so that the generic rank
of the mapping $(z,\eta)\mapsto (\bar{\Theta}_\beta(z,\eta))_{|\beta| \leq \ell_0}$ equals $r_M(p)$. Consider the
$\4M_0$ linear mapping $\6L\colon (\4M_0)^N\to (\4M_0)^{c_0}$ given by 
$$
\6L (a_1,\ldots,a_N):=
\left(\sum_{j=1}^Na_j\frac{\partial \bar{\Theta}_\beta}{\partial Z_j}\right)_{|\beta|\leq\ell_0},
$$
where $c_0={\rm Card}\, \{\beta \in \N^n: |\beta|\leq \ell_0 \}$. By Lemma~\ref{l:stanton} (i),
$\4S_0={\rm Ker}\, \6L$ and therefore, by the rank theorem, 
$N= \lambda_M(0)+{\rm dim}_{\4M_0}\, {\rm Im}\,\6L$. 
But the rank of the $\4M_0$-linear mapping $\6L$ coincides with the rank of the matrix
$\bigg(\displaystyle \frac{\partial \bar{\Theta}_\beta}{\partial Z_j}\bigg)_{1\leq j\leq N,\atop |\beta|\leq
\ell_0}$ with entries in the field $\4M_0$. This rank is exactly the generic rank of the
(germ of the) holomorphic map $(z,\eta)\mapsto (\bar{\Theta}_\beta(z,\eta))_{|\beta| \leq \ell_0}$. The proof of
Lemma~\ref{l:substitute} is complete. \end{proof}

A first consequence of Lemma~\ref{l:substitute} is that $r_M(p)$ is independent of the choice of normal coordinates. Another useful consequence of Lemma~\ref{l:substitute} is the following (see \cite{BR95} for the hypersurface case).

\begin{Lem}\label{l:anotherlemma} 
	Let $M\subset \C^N$ be a real-analytic generic submanifold. Then the functions
$\lambda_M$ and $r_M$ are constant on any connected component of $M$. 
\end{Lem} 
\begin{proof} Pick an arbitrary
point $p\in M$. We first note that  the definition of $\lambda_M$ implies that $\lambda_M$ is
lower semi-continuous. Next, let $r_M(p)$ be as defined in Lemma~\ref{l:substitute}. Since we may choose normal
coordinates for points $q\in M$ nearby $p$ in such a way that the mapping $\Theta$ depends real-analytically on
$q$, it follows that for all points $q$ sufficiently close to $p$, $r_M(q)\geq r_M(p)$. Hence by \eqref{e:kappa}, we
have that $\lambda_M(q)\leq \lambda_M(p)$ for all such $q$'s. This latter fact together with the lower
semi-continuity of $\lambda_M$ implies that $\lambda_M$ is constant in a neighborhood of $p$. Since the choice
of $p$ is arbitrary, we obtain the desired statement. 
\end{proof}

In what follows, given a positive integer $r$, we shall denote by $\6E_h^{r}$ the set of all germs through the
origin in $\C^r$ of holomorphically nondegenerate real-analytic CR submanifolds. Recall also that 
	given a positive integer $r$ and two germs of real submanifolds $(M,p)$ and $(M',p')$ in $\C^r$  we 
	 write $(M_1,p)\sim_h (M_2,p')$ if there exists a germ 
	of a biholomorphism of $\C^r$ at $p$ which maps the germ $(M,p)$ to $(M',p')$.

% It will be convenient to
% introduce the following definition.
% 
% \begin{Def}\label{d:holeq} 
% 	Given a positive integer $r$ and two germs of real submanifolds $M_1$ and $M_2$ in $\C^r$ at points $p$ 
% 	and $p'$ respectively, we shall write $(M_1,p)\sim_h (M_2,p')$ if there exists a germ 
% 	of a biholomorphism of $\C^r$ at $p$, sending the germ $(M_1,p)$ to $(M_2,p')$.
% \end{Def}

The following result provides a precise description of the holomorphic foliation on 
a (holomorphically degenerate) real-analytic CR submanifold.

\begin{Pro}\label{p:holstraight}
Let $M\subset \C^N$ be a connected real-analytic CR submanifold with $N\geq 2$. 
Then there exists a well-defined integer $\lambda_M \in \{0,\ldots,N-1\}$ 
and a closed proper real-analytic subvariety ${\Upsilon}_M\subset M$ such that 
$$M\setminus {\Upsilon}_M
=\{p\in M: (M,p)\sim_h (\C^{\lambda_M}\times \widetilde M,0),\ 
{\rm where}\ \widetilde M \in \6E_h^{N-\lambda_M}\}.$$
Furthermore, the real-analytic subvariety $\Upsilon_M$ is locally given by the intersection
of germs of complex-analytic subvarieties with $M$. 
\end{Pro}

\begin{proof} We first treat the case where $M$ is a generic submanifold of $\C^N$. Let $\lambda_M\in \{0,\ldots,N-1\}$ be the
integer defined by \eqref{e:lambda} and Lemma~\ref{l:anotherlemma}. Set $$\Omega^h_M:=\{p\in M: (M,p)\sim_h
(\C^{\lambda_M}\times \widetilde M,0),\ {\rm where}\ \widetilde M \in \6E_h^{N-\lambda_M}\}.$$ In what follows
we shall say that a point $p\in M$ satisfies property $(\spadesuit)$ if there exists
$Y_1,\ldots,Y_{\lambda_M}\in \4T_p$ and a sufficiently small neighborhood $W$ of $p$ in $\C^N$ such that these
$\lambda_M$ holomorphic vector fields are defined and linearly independent at every point $q\in W$. We first
note that if a point $p\in \Omega_M$ then clearly $p$ satisfies property $(\spadesuit)$. Conversely, if $p$
satisfies property $(\spadesuit)$, then by straigthening the flows of $Y_1,\ldots,Y_{\lambda_M}$, 
we see that
the germ $(M,p)$ is biholomorphically
equivalent to the germ at the origin of a submanifold of the form $\C^{\lambda_M}\times \widetilde M$ where
$\widetilde M$ is a germ through the origin of a real-analytic generic submanifold. Furthermore, $\widetilde M$
must necessarily be holomorphically nondegenerate since otherwise we could find $\lambda_M +1$ holomorphic
vector fields tangent to $M$ near $p$ and generically linearly independent in a neighbhorhood of $p$, which
contradicts the definition of $\lambda_M$.

Pick an arbitrary point $p\in M$. Since the ring $\4O_p$ is noetherian, it follows that $\4T_p$ is a finitely
generated submodule of $(\4O_p)^N$. Hence, there exists a connected neighborhood $U$ of $p$ in $\C^N$ and $r$
holomorphic vector fields $X_1,\ldots,X_r$ defined in $U$ such that $\4T_p$ is generated by the germs at $p$ of
the vector fields $X_1,\ldots,X_r$. In fact, we need a stronger property than that, and using Lemma~\ref{l:stanton} (ii)
and Oka's theorem (see e.g.\ \cite[Theorem~6.4.1]{Ho}), it is possible (after shrinking $U$ if necessary) to
assume that for every point $q\in M\cap U$, the germs at $q$ of the vector fields $X_1,\ldots,X_r$ generate
$\4T_q$. We also may assume that $U$ is chosen sufficiently small so that $M\cap U$ is connected. It is not
difficult to see that the generic rank of $(X_1,\ldots,X_r)$ over $U$ equals $\lambda_M$ as defined above.
Setting $\mu (q)={\rm Rk}\, (X_1(q),\ldots,X_r(q))$ for all $q\in M\cap U$, we now claim that
\begin{equation}\label{e:claim} \{q\in M\cap U: q\ {\rm satisfies}\ (\spadesuit)\}=\{q\in M\cap U: \mu
(q)=\lambda_M\}. 
\end{equation}

Note that we trivially have $\{q\in M\cap U: \mu (q)=\lambda_M\}\subset \{q\in M\cap U: q\ {\rm satisfies}\
(\spadesuit)\}$. Conversely, if $q\in M\cap U$ satisfies $(\spadesuit)$, there exists
$Y_1,\ldots,Y_{\lambda_M}\in \4T_q$ and a sufficiently small neighborhood $W\subset U$ of $q$ in $\C^N$ such
that these $\lambda_M$ holomorphic vector fields are defined and linearly independent at every point $q\in W$.
Since the germs at $q$ of $X_1,\ldots,X_r$ generate the $\4O_q$-module $\4T_q$ and since
$Y_1,\ldots,Y_{\lambda_M}$ are linearly independent at every point $q\in W$, it follows that the rank of
$(X_1,\ldots,X_r)$ equals $\lambda_M$ at every point of $W$ and hence at $q$. This shows the claim
\eqref{e:claim}.

Using the first part of the proof, we therefore obtain that $\Omega_M^h\cap U=\{q\in M\cap U: \mu
(q)=\lambda_M\}$ and hence $\Upsilon_M\cap U$ is given by the vanishing of a finite number of real-analytic
functions on $U$ (namely the restriction to $M\cap U$ of the minors of 
size $\lambda_M$ of the Jacobian matrix of $(X_1,\dots,X_r)$). This shows that $\Upsilon_M$ is a closed proper real-analytic subvariety of $M$ locally defined by the intersection of $M$ with a complex-analytic subvariety.

The case where $M$ is not generic follows from the generic case, after noticing that there exists an integer
$s\in \{1,\ldots,N-1\}$ such that for every point $p_0\in M$, the germ $(M,p_0)$ is locally biholomorphically
equivalent to a germ at the origin of a submanifold of the form $\{0\}\times M_1\subset \C^s \times \C^{N-s}$
with $M_1$ being a real-analytic generic submanifold in $\C^{N-s}$. We leave the remaining details to the
reader. The proof of Proposition~\ref{p:holstraight} is complete. \end{proof}

\begin{Rem}\label{r:lambda} It is clear that the integer $\lambda_M$ in Proposition~\ref{p:holstraight} is
unique and may be defined as follows. If $M$ is as in Proposition~\ref{p:holstraight}, for every point $p\in M$,
let $\6X_p\subset \C^N$ be the intrinsic complexification of $M$ at $p$ (see e.g.\ \cite{BERbook}). Then $\6X_p$
is the germ at $p$ of a complex submanifold (of smallest dimension) containing the germ of $M$ at $p$. Consider
the field $\6K_p$ of germs at $p$ of meromorphic functions in $\6X_p$ and $\6S_p$ the $\6K_p$ vector space of
all germs at $p$ of meromorphic vector fields of $\6X_p$ tangent to $M$. Then it follows from
Lemma~\ref{l:anotherlemma} that $M\ni p\mapsto {\rm dim}_{\6K_p}\, \6S_p$ is constant and is the desired integer
$\lambda_M$. \end{Rem}

\subsection{The real-algebraic case} We shall now establish the algebraic version of
Proposition~\ref{p:holstraight} when the submanifold $M$ is furthermore assumed to be real-algebraic. 
Analogously to the real analytic case, for a given  positive integer $r$, we denote by $\6E_a^{r}$ 
the set of all
germs through the origin in $\C^r$ of holomorphically nondegenerate real-algebraic CR submanifolds. We also 
recall that given  two germs of real submanifolds $(M,p)$ and $(M',p')$ in $\C^r$, we 
write $(M,p)\sim_a (M',p')$ if there exists a germ of an algebraic biholomorphism of $\C^r$ at $p$
which sends the germ $(M,p)$ to $(M',p')$.

% need
% the following variation of Definition~\ref{d:holeq}.
% 
% \begin{Def}\label{d:algeq} Given a positive integer $r$ and two germs of real submanifolds $M_1$ and $M_2$ in $\C^r$ at points $p$ and $p'$ respectively, we shall write $(M_1,p)\sim_a (M_2,p')$ if there exists a germ of an algebraic biholomorphism of $\C^r$ at $p$, sending the germ $(M_1,p)$ to $(M_2,p')$.
% \end{Def}

We are now ready to state the algebraic version of Proposition~\ref{p:holstraight}.

\begin{Pro}\label{p:algstraight}
Let $M\subset \C^N$ be a connected real-algebraic CR submanifold, $N\geq 2$. Let $\lambda_M$ and $\Upsilon_M$ be the associated integer and real-analytic subvariety of $M$ given by Proposition~{\rm \ref{p:holstraight}}. Then $\Upsilon_M$ is in fact a proper real-algebraic subvariety of $M$ and the following holds:
\begin{equation}\label{e:oka}
M\setminus {\Upsilon}_M=\{p\in M: (M,p)\sim_a (\C^{\lambda_M}\times \widetilde M,0),\ {\rm where}\ \widetilde M \in \6E_a^{N-\lambda_M}\}.
\end{equation}
Furthermore, the real-algebraic subvariety $\Upsilon_M$ is locally given by the intersection
of germs of complex-algebraic subvarieties with $M$.
\end{Pro}

In order to prove Proposition~\ref{p:algstraight}, we shall need the following observation.
\begin{Lem}\label{l:observe} Let $M$ be a germ of a real-algebraic CR submanifold in $\C^{N}$ and $M'$ be a germ of a real-analytic CR  submanifold in $\C^{N'}$, both through the origin, with $1 \leq N'<N$. Assume that \begin{equation}\label{e:holeq}
(M,0)\sim_h (M'\times \C^{N-N'},0).
\end{equation}
Then there exists a germ of a real-algebraic CR submanifold $\widehat M\subset \C^{N'}$ through the origin such that
\begin{equation}\label{e:algeq}
(M,0)\sim_a (\widehat M\times \C^{N-N'},0).
\end{equation}
\end{Lem}

\begin{proof}[Proof of Lemma~{\rm \ref{l:observe}}] First note that if $M$ is not generic in $\C^N$, then there exists a positive integer $r\in \{1,\ldots,N-1\}$ and a real-algebraic generic submanifold $M_1\subset \C^{N-r}$ through the origin such that $(M,0)\sim_a(\{0\}\times M_1,0)$ (see e.g.\ \cite{BERbook}). From this fact, we see that we may assume in what follows that $M$ is generic in $\C^N$.

Let $n$ be the CR dimension of $M$ and $d$ its codimension. Choose normal coordinates $Z=(Z_1,\ldots,Z_N)=(z,\eta)\in \C^n \times \C^d$ as in Section~\ref{sss:choose}, where $\Theta\colon (\C^{n+N},0)\to (\C^d,0)$ is an algebraic holomorphic map of its arguments.

We first note that, by a usual straightening argument, \eqref{e:holeq} is equivalent to say that there exists a neighborhood $V$ of $0$ in $\C^N$ and $N-N'$ holomorphic vector fields $L_1,\ldots,L_{N-N'}$ tangent to $M\cap V$ such that these $N-N'$ vector fields are linearly independent at every point of $V$. Let $\phi^1(t,Z)$ be the complex flow of the vector field $L_1$, flow that is defined for $(t,Z) \in \C \times \C^N$ sufficiently close to the origin. Recall that for sufficiently small $t$, $(\C^N,0)\ni Z\mapsto \phi^1(t,Z) \in (\C^N,0)$ is a germ at $0$ of a biholomorphism sending $(M,0)$ to itself. Furthermore, writing 
$L_1=\sum_{j=1}^Na_j(Z)\frac{\partial}{\partial Z_j}$, we know that $\phi^1=(\phi^1_1,\ldots,\phi^1_N)$ satisfies
\begin{equation}\label{e:der}
\frac{\partial \phi^1_j}{\partial t}(t,Z)=a_j(\phi^1(t,Z)),\quad j=1,\ldots,N,\quad {\rm and}\ \phi^1(0,Z)=Z,
\end{equation}
for $(t,Z) \in \C \times \C^N$ sufficiently close to the origin. Since $L_1$ is tangent to $M$ (near $0$), using the notation defined in  \eqref{e:expansion} and Lemma~\ref{l:stanton} (i), we have
\begin{equation}\label{e:derder}
\sum_{j=1}^Na_j(\phi^1(t,Z))\frac{\partial \bar{\Theta}_\beta}{\partial Z_j}(\phi^1(t,Z))=0,\quad \forall \beta \in \N^n, 
\end{equation}
for $(t,Z) \in \C \times \C^N$ sufficiently close to the origin. Combining \eqref{e:der} and \eqref{e:derder}, we obtain that near the origin in $\C^{N+1}$
$$\frac{\partial}{\partial t} \bigg( \bar{\Theta}_\beta (\phi^1(t,Z))\bigg)\equiv 0,\quad \forall \beta \in \N^n.$$
We therefore have the following identity (that is contained, in the hypersurface case, in the statement of \cite[Proposition 5.2]{BR95})
\begin{equation}\label{e:trick}
\bar{\Theta}_\beta (\phi^1(t,Z))=\bar{\Theta}_\beta (Z),\quad \forall \beta \in \N^n,
\end{equation}
for all $(t,Z) \in \C \times \C^N$ sufficiently close to the origin. In what follows, we may assume without loss of generality that $a_1(0)\not =0$. Consider now the $\C^N$-valued holomorphic mapping $\Psi$ defined in a neighbhorhood of the origin in $\C^N$  by $\Psi (w_1,\ldots,w_N)=\phi^1(w_1,0,w_2,\ldots,w_N)$. It is a standard fact that $\Psi$ is a local biholomorphism at the origin and that $\Psi_*(\frac{\partial}{\partial w_1})=L_1$. Hence, since $L_1$ is tangent to $M$ near $0$, the vector field $\frac{\partial}{\partial w_1}$ is tangent to the germ of the real-analytic generic submanifold $\Psi^{-1}(M)$. Therefore the germ of $\Psi^{-1}(M)$ at $0$ is of the form $\widetilde M \times \C$ where $\widetilde M$ is a germ at $0$ of a real-analytic generic submanifold in $\C^{N-1}$. We now claim that $\widetilde M$ is in fact real-algebraic and that $(\widetilde M \times \C,0)\sim_a (M,0)$. To prove the claim, we note that it follows from \eqref{e:trick} that for all $\beta \in \N^n$, $\bar{\Theta}_\beta (\Psi (w))=\bar{\Theta}_\beta (0,w_2,\ldots,w_N)=:C_\beta (w)\in (\6N \{w\})^d$. The system of algebraic equations
\begin{equation}\label{e:systartin}
\bar{\Theta}_\beta (Z)=C_\beta (w)
\end{equation}
has a convergent solution $Z=\Psi (w)$ and therefore, from an approximation theorem due to Artin \cite{A69}, there exists an algebraic solution $\widehat \Psi \colon (\C^N,0)\to (\C^N,0)$ of \eqref{e:systartin} that agrees with $\Psi$ up to order one at $0$. Hence $\widehat \Psi$ is a local algebraic biholomorphism. Furthermore, since $\Psi$ sends $(\widetilde M \times \C,0)$ to $(M,0)$ and since for all $\beta \in \N^n$
$$\bar{\Theta}_\beta (\Psi (w))=\bar{\Theta}_\beta (\widehat \Psi (w)),$$
it follows from \cite[Lemma~14.3]{BMR} that $\widehat \Psi$ sends also $(\widetilde M \times \C,0)$ to $(M,0)$, which proves the claim. 

Since $(\widetilde M \times \C,0)\sim_a (M,0)$,  there exists a neighborhood $W$ of $0$ in $\C^{N-1}$ and $N-N'-1$ holomorphic vector fields tangent to $\widetilde M\cap W$ such that these $N-N'-1$ vector fields are linearly independent at every point of $W$. We can therefore apply the above reasoning to the real-algebraic generic submanifold $\widetilde M\subset \C^{N-1}$ and, by induction, we reach the desired conclusion.
\end{proof}

\begin{proof}[Proof of Proposition~{\rm \ref{p:algstraight}}] We note that if $\lambda_M=0$ (i.e.\ $M$ is holomorphically nondegenerate), there is nothing to prove, and therefore assume that $\lambda_M>0$. Lemma~\ref{l:observe} leads immediately to \eqref{e:oka}. It remains to show that $\Upsilon_M$ is a  real-algebraic subvariety of $M$. In order to show this, we pick an arbitrary point $p\in M$ choose associated normal coordinates, in which $p=0$. 
We note that as in the proof of  
Proposition~\ref{p:holstraight}, we can choose a positive integer $\ell_0$ such 
that in a neighbourhood $U$ of $0$, a germ at $q\in U$ of a holomorphic vector field is tangent to $M$ if and 
only if $\sum_{j=1}^N\frac{\partial {\bar \Theta}_{\beta}}{\partial Z_j} (Z)\, a_j(Z) = 0$ near $q$ for all $\beta$ with $|\beta|\leq \ell_0$. Thus, we have realized the
holomorphic vector fields tangent to $M\cap U$ as a subsheaf of $\4O^N|_U$ given by the relations between 
the $(\frac{\partial {\bar \Theta}_{\beta}}{\partial Z_1},\ldots,\frac{\partial {\bar \Theta}_{\beta}}{\partial Z_N})$ with $|\beta|\leq \ell_0$. Oka's Theorem then implies that we can find a finite number of 
holomorphic generators of $\4T_0$ which also generate $\4T_q$ for $q$ near $0$. 

We  now claim that we can actually choose these generators algebraic if all the ${\bar \Theta}_\beta$
are algebraic; the remainder of the proof is then verbatim to the proof of Proposition~\ref{p:holstraight}.

First note that the sheaf of algebraic functions on $\CN$ is coherent. This can be seen by 
applying the Weierstrass division theorem for algebraic power series in the proof of Oka's Theorem
as found in e.g. \cite{Ho}. In particular, the sheaf of {\em algebraic} relations between the 
$(\frac{\partial {\bar \Theta}_\beta}{\partial Z_1},\ldots \frac{\partial {\bar \Theta}_\beta}{\partial Z_N})$ for $|\beta|\leq \ell_0$ is locally finitely generated, i.e. shrinking $U$ if necessary, there exist algebraic vector fields $(X_1,\ldots, X_r)$ defined over $U$ and tangent to $M$, such that every germ at a point $q\in M\cap U$ of an algebraic vector field tangent to  $M$ near $q$  can be written as an algebraic linear combination of $X_1,\dots, X_r$ .

Now for every point $q\in U$, the ring of convergent power series $\cps{x-q}$ centered at $q$  is a flat algebra over the ring of germs at $q$ of algebraic functions  $\6N \{x-q\}$. We can thus 
apply the ``equational criterion for flatness'' (see e.g.  \cite[Theorem 7.6]{Matsumura})
to see that every $X\in \4T_q$ can (in the sense of germs at $q$) 
be written as a linear combination with coefficients in $\cps{x-q}$ of holomorphic vector fields tangent to $M$ with algebraic coefficients in $\6N \{x-q\}$. Since we have already observed that all germs at $q$ of algebraic vector fields tangent to $M$ are generated by $X_1,\dots , X_r$, we see that  the claim is proved and, as noted before, this finishes the proof of Proposition~\ref{p:algstraight}.
\end{proof}

\begin{Rem}
	The authors thank Clemens Bruschek for pointing out the simple proof by applying the 
	flatness criterion in the second half of the proof of Proposition~\ref{p:algstraight}.
\end{Rem}

\begin{Rem}\label{r:comparebrz}(i) In \cite{BRZ2}, the authors have considered another proper real-algebraic subvariety $\widetilde \Upsilon_M$ attached to any connected real-algebraic CR submanifold $M\subset \C^N$. This subvariety $\widetilde \Upsilon_M$ is defined as follows: $M\setminus \widetilde \Upsilon_M:=\{p\in M: (M,p)\sim_a (\C^{\lambda_M}\times \widetilde M,0),\ {\rm where}\ \widetilde M \in \widetilde{\6E}_a^{N-\lambda_M}\},$ where $\widetilde{\6E}_a^{N-\lambda_M}$ denotes the set of germs of all real-algebraic finitely nondegenerate real-analytic CR submanifolds in $\C^{N-\lambda_M}$ through the origin (see e.g.\ \cite{BERbook, BRZ1} for the definition). Since we always have the strict inclusion $\widetilde{\6E}_a^{N-\lambda_M}\subset {\6E}_a^{N-\lambda_M}$, the subvariety $\Upsilon_M$ given by Proposition~\ref{p:algstraight} is in general strictly smaller than $\widetilde \Upsilon_M$.\\
(ii) The reader should observe that the real-algebraic subvariety $\Upsilon_M$ given by Proposition~\ref{p:algstraight} can also be defined as follows:
$M\setminus {\Upsilon}_M$ consists of all points $p$ in $M$ for which there exists
an integer $k$, $0\leq k\leq N-1$, such that $(M,p)\sim_a (\C^{k}\times \widetilde M,0),\ {\rm where}\ \widetilde M \in \6E_h^{N-k}$.
% $$M\setminus {\Upsilon}_M=\{p\in M: \exists k\in \{0,\ldots,N-1\}\ {\rm s.t.}\ (M,p)\sim_a (\C^{k}\times \widetilde M,0),\ {\rm where}\ \widetilde M \in \6E_h^{N-k}\}.$$
Indeed, note that if there exists $k\in \{0,\ldots,N-1\}$ such that $(M,p)\sim_h (\C^{k}\times \widetilde M,0)$ where $\widetilde M \in \6E_h^{N-k}$, then  we must necessarily have $k=\lambda_M$ in view of the definition of $\lambda_M$.
\end{Rem}

\section{Proof of Theorem~\ref{t:main2holfol}}\label{s:last}

By definition $\Sigma_M=\Sigma_M^1\cup \Sigma_M^2$ where $\Sigma^1_M$ is the set of points that are not of constant orbit dimension and $\Sigma^2_M$ is the set of points that are not regular for the foliation on $M$. By Proposition~\ref{p:algstraight}, $\Sigma^2_M$ is a closed proper real-algebraic subvariety of $M$ and since $\Sigma^1_M$ possesses clearly also the same property (see e.g.\ \cite{BERbook}), conclusion (i) of the theorem follows.

To prove (ii), we may assume that $M$ is a generic submanifold in $\C^N$ 
since the non-generic case can easily be reduced to the generic
case. Let $p\in M \setminus \Sigma_M$. Let also $M'\subset \C^N$ be another real-algebraic generic
submanifold, $p'\in M'$ and let $h\colon (\C^N,p)\to (\C^N,p')$ be a local biholomorphic map sending $M$ to $M'$. 
If $M$
is holomorphically nondegenerate, the desired conclusion follows immediately from Theorem~\ref{t:technic}.

If not, there exists an integer $k\in \{1,\ldots,N-1\}$ such that $(M,p)\sim_a (\widetilde M \times \C^k,0)$ where
$\widetilde M$ is a holomorphically nondegenerate real-algebraic generic submanifold through the origin in $\C^{N-k}$.
Furthermore, since $(M',p')\sim_h (M,p)$, it follows from Lemma~\ref{l:observe} that 
$(M',p')\sim_a (\widetilde M'\times\C^k,0)$, where $\widetilde M'$ is a real-algebraic generic submanifold through the origin in $\C^N$ which is also necessarily holomorphically nondegenerate. 
In order to prove Theorem~\ref{t:main2holfol}, we may therefore assume that
$(M,p)=(\widetilde M \times \C^k,0)$ and that $(M',p')= (\widetilde M'\times \C^k,0)$. We also write the mapping $h\colon
(\C^{N-k}_{\widetilde Z}\times \C^k_{\widehat Z},0)\to (\C^{N-k}_{\2Z'}\times \C^k_{\7Z'},0)$, where $h(Z)=h(\2 Z,\7
Z)=(h_1(\2 Z,\7 Z),h_2(\2 Z,\7 Z))\in \C^{N-k}\times \C^k$. 
We claim that $h_1$ is independent of $\7Z=(\7Z_1,\ldots,\7Z_k)$. Indeed, consider
the holomorphic vector $V=h_*(\frac{\partial }{\partial \7Z_j})$ where $j=1,\ldots,k$. We have 
$$V=\frac{\partial h_1}{\partial \7Z_j}(h^{-1}(Z'))\cdot \frac{\partial}{\partial
\2Z'}+\frac{\partial h_2}{\partial \7Z_j}(h^{-1}(Z'))\cdot \frac{\partial}{\partial \7Z'}.$$ 
Since the vector field
$\frac{\partial }{\partial \7Z_j}$ is tangent to $\2M \times \C^k$ near 0, $V$ is tangent to $\2M'\times \C^k$ near
$0$. This implies that for every $\7Z'\in \C^k$ sufficiently close to the origin, the holomorphic vector field in
$\C^{N-k}$ $$\frac{\partial h_1}{\partial \7Z_j}(h^{-1}(\2Z',\7Z'))\cdot \frac{\partial}{\partial \2Z'}$$ is tangent to
$\2M'$ near $0$. Since $\2M'$ is holomorphically nondegenerate, we must necessarily have $\frac{\partial h_1}{\partial
\7Z_j}(h^{-1}(\2Z',\7Z'))\equiv 0$ near $0\in \C^N$ for all $j=1,\ldots,k$, which proves the claim. 

We can thus write $h(\2Z,\7Z)=(h_1(\2Z),h_2(\2Z,\7Z))$ where $h_1\colon (\C^{N-k},0)\to (\C^{N-k},0)$ is a local
biholomorphism sending $(\2M,0)$ to $(\2M',0)$. 
Since $\2M$ is holomorphically nondegenerate and because the local CR orbits
of $\2M$ must  also be of constant dimension in a neighbhorhood of $0$, we may apply Theorem~\ref{t:technic} to conclude
that for every integer $\ell$ there exists a local algebraic biholomorphism $h_1^\ell \colon (\C^{N-k},0)\to (\C^{N-k},0)$
sending $(\2M,0)$ to $(\2M',0)$ that agrees with $h_1$ up to order $\ell$ at $0$. For every integer $\ell$, define
$h_2^\ell$ to be the $\ell$-th order Taylor polynomial of $h_2$ at $0$. Then the local algebraic biholomorphism
$h^\ell:=(h_1^\ell,h_2^\ell)$ satisfies all the required conditions. The proof of Theorem~\ref{t:main2holfol} is complete.

\end{document}